\documentclass[10pt,reqno]{amsart}
\usepackage[b5paper,top=1.2in,left=0.9in]{geometry}
\usepackage{verbatim}
\usepackage{amssymb}
\usepackage{amsmath}
\usepackage{graphicx}
\usepackage{appendix}
\usepackage{color}
\usepackage{amsthm}
\usepackage{tikz}
\usetikzlibrary{arrows}

\renewcommand{\d}{\mathrm{d}}
\newcommand{\D}{\mathrm{D}}
\renewcommand{\i}{\mathrm{i}}

\newtheorem{Thm}{Theorem}[section]
\newtheorem{Lem}[Thm]{Lemma}
\newtheorem{Prop}[Thm]{Proposition}
\newtheorem{Cor}[Thm]{Corollary}
\newtheorem{Rem}[Thm]{Remark}
\newtheorem{Def}[Thm]{Definition}
\newtheorem{Con}[Thm]{Conjecture}
\newtheorem{Ex}[Thm]{Example}
\newtheorem{Fact}[Thm]{Fact}
\newtheorem{Nota}[Thm]{Notation}

\newtheorem*{MThm}{Main Theorem}

\def\R{\mathbb{R}}
\def\Q{\mathbb{Q}}
\def\N{\mathbb{N}}
\def\C{\mathbb{C}}
\def\Z{\mathbb{Z}}
\def\T{\mathbb{T}}

\def\to{\longrightarrow}

\def\cA{\mathcal{A}}
\def\cB{\mathcal{B}}

\def\cD{\mathcal{D}}
\def\cE{\mathcal{E}}

\def\cH{\mathcal{H}}

\def\cM{\mathcal{M}}

\def\cP{\mathcal{P}}

\def\cU{\mathcal{U}}

\def\cW{\mathcal{W}}
\def\a{\alpha}
\def\b{\beta}

\def\c{\gamma}
\def\D{\Delta}

\def\d{\delta}

\def\l{\lambda}
\def\L{\Lambda}

\def\t{\tau}
\def\W{\Omega}
\def\w{\omega}
\def\ze{\zeta}

\def\fb{\mathfrak{b}}
\def\sl{\mathfrak{sl}}
\def\g{\mathfrak{g}}

\def\fh{\mathfrak{h}}
\def\fq{\mathfrak{q}}

\def\ox{\otimes}
\def\o+{\oplus}
\def\bo+{\bigoplus}
\def\x{\times}
\def\p[#1,#2]{\phi_{#1,#2}}
\def\til[#1]{\widetilde{#1}}
\def\what[#1]{\widehat{#1}}

\def\ba{\mathbf{a}}
\def\bb{\mathbf{b}}
\def\bc{\mathbf{c}}
\def\bC{\mathbf{C}}
\def\bu{\mathbf{u}}
\def\bv{\mathbf{v}}

\def\bx{\mathbf{x}}

\def\be{\mathbf{e}}
\def\bE{\mathbf{E}}
\def\bf{\mathbf{f}}
\def\bF{\mathbf{F}}

\def\bi{\mathbf{i}}

\def\bK{\mathbf{K}}

\def\bp{\mathbf{p}}

\def\bU{\mathbf{U}}

\def\bii{\underline{\mathbf{i}}}

\def\z[#1]{z_{#1}}

\def\oo{\infty}
\def\=>{\Longrightarrow}
\def\inj{\hookrightarrow}

\def\<{\langle}
\def\>{\rangle}
\def\corr{\longleftrightarrow}
\def\^{\wedge}
\def\+{\dagger}
\def\iff{\Longleftrightarrow}

\def\sub{\subset}
\def\inv{^{-1}}
\def\dis{\displaystyle}
\def\over[#1]{\overline{#1}}
\def\vec[#1]{\overrightarrow{#1}}
\def\mat[#1, #2]{\left[\begin{array}{ccccc}#1\end{array}\left|\begin{array}{c}#2\end{array}\right.\right]}
\def\xto[#1]{\xrightarrow{#1}}
\def\dd[#1,#2]{\frac{d#1}{d#2}}
\def\del[#1,#2]{\frac{\partial #1}{\partial #2}}
\def\pd{\partial}
\def\Facts[#1]{\begin{Fact}\mbox{}\begin{itemize}#1\end{itemize}\end{Fact}}
\def\Notation[#1]{\begin{Nota}\mbox{}\begin{itemize}#1\end{itemize}\end{Nota}}
\def\tab{\;\;\;\;\;\;}

\newcommand{\veca}[2][cccccccccccccccccccccccccccccccccccccccccc]{\left(\begin{array}{#1}#2 \\ \end{array} \right)}

\newcommand{\Eq}[1]{\begin{align}#1\end{align}}
\newcommand{\Eqn}[1]{\begin{align*}#1\end{align*}}
\newcommand{\case}[2][lllllllllllllllllllllllllll]{\left\{\begin{array}{#1}#2 \\ \end{array}\right.}

\begin{document}
\title[Positive representations of $\cU_q(\g_\R)$]{Positive representations of split real simply-laced quantum groups}
\dedicatory{Dedicated to Igor Frenkel on his 60th birthday}
\author{Ivan C.H. Ip}
\thanks{Department of Mathematics, Hong Kong University of Science and Technology, Hong Kong}
\thanks{Email: ivan.ip@ust.hk}

\numberwithin{equation}{section}
\maketitle

\begin{abstract}
We construct the positive principal series representations for $\cU_q(\g_\R)$ where $\g$ is of simply-laced type, parametrized by $\R_{\geq 0}^r$ where $r$ is the rank of $\g$. We describe explicitly the actions of the generators in the positive representations as positive essentially self-adjoint operators on a Hilbert space, and prove the transcendental relations between the generators of the modular double. We define the modified quantum group $\bU_{\fq\til[\fq]}(\g_\R)$ of the modular double and show that the representations of both parts of the modular double commute weakly with each other, there is an embedding into a quantum torus algebra, and the commutant contains its Langlands dual. 
\end{abstract}

\setcounter{tocdepth}{1}
\tableofcontents
\newpage
%=======================================================================
\section{Introduction}\label{sec:intro}

To any finite-dimensional complex simple Lie algebra $\g$, Drinfeld \cite{D} and Jimbo \cite{J} defined a remarkable Hopf algebra $\cU_q(\g)$ known as the quantum group. Its representation theory has evolved into an important area with many applications in various fields of mathematics and physics (cf. \cite{EGGHS} and reference therein).

In classical Lie theory, one is interested in certain real subalgebras of $\g$ known as \emph{real forms}, two important cases being $\g_c$ corresponding to compact groups (e.g. $SU(n)$), and $\g_\R$ corresponding to split real groups (e.g. $SL(n,\R)$). The finite-dimensional representation theory in the compact case is well behaved, and it is generalized nicely to the corresponding quantum group $\cU_q(\g_c)$. However, on the contrary, representation theory in the split real case is much more complicated as was shown by the monumental works of Harish-Chandra. Its generalization to the quantum group level - involving self-adjoint operators on Hilbert spaces -- is physically more relevant, but is still largely open due to various analytic difficulties coming from noncompactness and the use of unbounded operators.

\subsection{Positive representations}
In this second paper of the series, we give the construction of the \emph{positive principal series representations}, or \emph{positive representations} for short, of the modular double $\cU_{q\til[q]}(\g_\R)$ of the split real quantum group for simply-laced $\g$, generalizing our first work \cite{FI} where the positive representations for the modular double $\cU_{q\til[q]}(\sl(n,\R))$ are constructed for the first time. This result strengthens the perspectives discussed in \cite{FI} for a new direction of representation theory of split real quantum groups since the discovery of the concept of the modular double for quantum groups \cite{Fa}, and in the case of $\cU_{q\til[q]}(\sl(2,\R))$, the special class of representations studied by Ponsot--Teschner \cite{PT1}.

Let us specialize the quantum parameter $q$ to be $q=e^{\pi \bi b^2}\in\C$, where $b^2\in\R\setminus\Q$ and $0<b<1$. Let $E_i,F_i,K_i$ be the generators of $\cU_q(\g_\R)$ with the standard quantum relations. Similarly let $\til[E_i],\til[F_i],\til[K_i]$ be the generators of $\cU_{\til[q]}(\g_\R)$ by replacing $b$ with $b\inv$, where $\til[q]=e^{\pi \bi b^{-2}}$. Furthermore, denote the rescaled variables by
\Eq{\label{rescale}
e_i:=2\sin(\pi b^2)E_i, \tab f_i:=2\sin(\pi b^2)F_i}and similarly for $\til[e_i]$ and $\til[f_i]$ with $b$ replaced by $b\inv$.

In this paper we construct the \emph{positive representations} for the split real form $\cU_q(\g_\R)$, which have the following remarkable properties:
\begin{itemize}
\item[(i)] the generators $e_i,f_i,K_i^{\pm1}$ and $\til[e_i],\til[f_i],\til[K_i]^{\pm1}$ are represented by positive essentially self-adjoint operators on a Hilbert space;
\item[(ii)] the generators satisfy the transcendental relations (as positive operators on the same domain)
\Eq{\label{transcendenal}e_i^{\frac{1}{b^2}}=\til[e_i],\tab f_i^{\frac{1}{b^2}}=\til[f_i],\tab K_i^{\frac{1}{b^2}}=\til[K]_i}
\end{itemize}
Furthermore, let $\be_i,\bf_i,\bK_i^{\pm 1}$ (similarly for the tilde variables) be the modified generators obtained from multiplication by certain factors of the $K_i$ (see Definition \ref{modified}). Then we also obtain the compatibility of the positive representations with the modular double $\bU_{\fq\til[\fq]}(\g_\R)$:
\begin{itemize}
\item[(iii)]  the modified generators $\be_i,\bf_i,\bK_i^{\pm 1}$ commute weakly with $\til[\be]_i,\til[\bf]_i,\til[\bK]_i^{\pm 1}$.
\end{itemize}

In the case of $SL(n,\R)$, there are two natural coordinate systems on the totally positive unipotent semi-subgroup $U_{>0}^+$. These are the Lusztig's data parametrized by a given choice of reduced expression of the longest element $w_0$, and the cluster coordinates given by the determinants of the square submatrices. In this paper, we choose the Lusztig data as the coordinates of the totally positive unipotent subspace, since the exchange relations for the Lusztig coordinates are more explicit than the cluster coordinates given by the generalized minors for arbitrary type $\g$. The transformation between the coordinates corresponding to different reduced expressions of the longest element $w_0$ can be written explicitly in Lusztig coordinates. 

\subsection{Main results}

The main results of the paper are the following: 
\begin{MThm} There exists a family of irreducible representations $\cP_\l\simeq L^2(\R^N)$ of $\cU_q(\g_\R)$ and its (modified) modular double $\bU_{\fq\til[\fq]}(\g_\R)$, where $N=\dim (U_{>0}^+)=l(w_0)$, parametrized by $\l\in\R^r$ where $r=rank(\g)$, satisfying properties (i)--(iii) above. 
\end{MThm}

More precisely, for every reduced expression for $w_0$ we can construct, explicitly, the positive representations. For each change of words of $w_0$, we establish the following unitary transformation, so that in particular the family of positive representation is \emph{canonical}, where it is independent of choice of the reduced expression of $w_0$.

\begin{Thm}\label{Thmtrans}For a change of words $(...,i,j,i,...)\corr (...,j,i,j,...)$ with the corresponding variables $(...u,v,w...)$ of the Hilbert space $L^2(\R^N)$, the generator of $\cU_q(\g_\R)$ acting as an operator $X$ on $L^2(\R^N)$ is transformed unitarily by
\Eq{X\corr \Phi X\Phi\inv,}
where
\Eq{\Phi = T\circ M,}
is a unitary transformation on $L^2(\R^N)$ with
\Eq{
M=g_b(e^{\pi b(2p_w-2p_u+u-v+w)})\circ g_b^*(e^{\pi b(2p_w-2p_u-u+v-w)})} a unitary operator expressed in terms of the functional calculus applied to the quantum dilogarithm function $g_b$ and its complex conjugate, while $T$ is a linear transformation on the variables and has determinant 1.
\end{Thm}

Since the transformation is unitary, it suffices to show the commutation relations, the positivity and the transcendental relations for a specific reduced expression of $w_0$. In particular, by choosing a ``good" reduced expression for $w_0$, the above properties follow immediately.

On the other hand, by choosing the expression for $w_0$ in a particular way, we have the folloing theorem.
\begin{Thm} The positive representations for type $D_n$, $n\geq 4$ and $E_6$, $E_7$, $E_8$ are constructed explicitly. In particular, the minimal principal series representations for the \emph{classical} $\cU(\g_\R)$ in terms of finite difference operators can be read off from the expressions.
\end{Thm}
The general expression for type $D_n$ is given in Theorem \ref{ThmDn}, while the explicit expressions in type $E_n$ can be found in the Appendix (reproduced from \cite{Ip3}).

Furthermore, as in the type $A_n$ case \cite{FI}, by using the modified version $\bU_{\fq\til[\fq]}(\g_\R)$ of the modular double we have the following important properties.
\begin{Thm}\label{Thmqtori}
We have an embedding
\Eq{\bU_{\fq}(\g)\inj\T_{\fq}^{N}.}
of the modified quantum group into the Laurent polynomials generated by $N$ $q$-tori, where $N=l(w_0)=dim(U_{>0}^+)$ with certain coefficients. Furthermore, representing the $q$-tori using the canonical position and momentum operators, we recover the positive representation of the modular double $\bU_{\fq\til[\fq]}(\g_\R)$.
\end{Thm}
\begin{Thm} \label{ThmLL} The generators of the Langlands dual group $\bU_{\til[\fq]}({}^L\g_\R)$, obtained by adjoining to $\bU_{\til[\fq]}(\g_\R)$ the fractional powers of its Cartan generators, commute weakly with the generators of $\bU_{\fq}(\g_\R)$ under the positive representations.
\end{Thm}

Finally, in the positive representations of $\cU_{q\til[q]}(\sl(2,\R))$, it is shown in \cite{PT2} that the representations corresponding to the parameters $\l$ and $-\l$ are equivalent under certain transformations involving multiplications by the quantum dilogarithms. In the general case there is a natural action of the Weyl group $W$ on the parameters $\l\in\R^r$, where $r$ is the rank of $\g$. Then we have the following result:
\begin{Thm}\label{RW} The positive representations corresponding to the parameters $\l$ and $w(\l)$ where $w\in W, \l\in\R^r$ are unitarily equivalent. In particular the positive representations are parametrized by $\l\in\R_{\geq 0}^r\simeq \R_{\geq 0}P^+$,  the cone of the $\R_{\geq0}$-span of dominant weights.
\end{Thm}

\subsection{Recent progress and future directions}

There are several problems yet to be answered. A natural question is whether these representations can be generalized to $\cU_q(\g_\R)$ of arbitrary type. In the third paper of this series \cite{Ip4}, we construct the positive representations for the non-simply-laced type, where it turns out that the transcendental relations play a crucial role relating its Langlands dual. 

In \cite{Ip} we prove a Peter--Weyl type theorem in the case of $\cU_{q\til[q]}(\sl(2,\R))$. It is shown that under the left and right regular representations of $\cU_{q\til[q]}(\sl(2,\R))$, the Hilbert space $L^2(SL_q^+(2,\R))$, suitably defined using the Gelfand--Naimark--Segal (GNS) construction, decomposes into a direct integral of tensor products of the positive representations $\cP_\l$, with the Plancherel measure given by the quantum dilogarithms. In general, the Haar functional needed to construct the $L^2$ space structure $L^2(G_q^+)$ of (the modular double of) the quantized function space $F_{q\til[q]}(G^+)$ is suggested in \cite{Ip2}. Together with the remark after Theorem \ref{transLi}, one can ask the following
\begin{Con}\label{conL2}
The space $L^2(G_q^+)$ is decomposed into (a direct integral) of tensor products of the positive representations as the left and right regular representations of $\cU_{q\til[q]}(\g_\R)$:
$$L^2(G_q^+)\simeq \int_{\R_{\geq 0}^N}^{\o+} \cP_\l\ox \cP_\l d\mu(\l)$$
for some Plancherel measure $d\mu$.
\end{Con} 

In \cite{PT2} it is shown that the class of positive representations $\cP_\l$ for $\cU_{q\til[q]}(\sl(2,\R))$ is closed under the tensor product (as a direct integral), with the same Plancherel measure appearing in the Peter--Weyl theorem. In particular the positivity and transcendental relations are preserved under the tensor product. A natural question for general $\g$ is then the following
\begin{Con}\label{contensor}The class of positive representations is closed under tensor product (as a direct integral):
$$\cP_\a\ox \cP_\b\simeq \int_{\R_{\geq 0}^N}^{\o+} M_{\a\b}^\c \ox \cP_\c d\mu(\c)$$
for some measure $d\mu$  and some multiplicity space $M_{\a\b}^\c$.
\end{Con}

The research direction from the problem of the closure of $\cP_\l$ under taking tensor product leads to the study of its \emph{positive Casimir} operators \cite{Ip7} and the semiclassical limit of the Clebsch--Gordan coefficients \cite{Ip6}. A promising direction is the recent results of the cluster algebraic realization of $\cU_q(\g)$ via the positive representations in terms of certain quantum torus algebra \cite{Ip8} related to quantization of higher Teichm\"uller theory \cite{FG}. It was pointed out by Schrader and Shapiro that in this notion the positive Casimir elements can be represented as certain monodromies around punctures of some surface, which (at least in type $A_n$) has a meaning of quantum Hamiltonian operators of certain $q$-Toda integrable systems, and is expected to be simultaneously diagonalizable \cite{SS}.

In a separate publication \cite{Ip5}, we constructed the universal $R$ operator for arbitrary type $\cU_q(\g_\R)$, giving it the braiding structure. Together with Conjecture \ref{contensor} about the tensor product, the class of positive representations will form a new \emph{``continuous" braided tensor category}, and it is envisioned in \cite{FI} that this will lead to future applications comparable to the development of the finite-dimensional representation theory of compact quantum groups since its discovery by Drinfeld and Jimbo over 30 years ago.

The properties of being \emph{positive self-adjoint} become central in further analysis, where they are needed to define complex powers of the generators in order to do harmonic analysis on the $C^*$-algebraic level, its relation to multiplier Hopf algebra, the construction of the universal $R$ operator and a notion of a continuous canonical basis \cite{Ip1}. On the other hand, the \emph{transcendental relations} are also important in understanding the interplay between the modular double, which is again united in the $C^*$-algebraic setting. Furthermore, these properties allow us to employ the powerful \emph{quantum dilogarithm function} for various constructions. These aspects distinguish positive representations as a remarkable class from other integrable representations, and have been partially investigated in \cite{Ip, Ip2,Ip5}.

Finally, we would like to remark that there are other ways to deform the principal series of representations even associated with the same minimal parabolic subalgebra $\fb_\R^+\subset\g_\R$. For example, a class of representations of $\cU_{q}(\g)$ has been constructed by Gerasimov et al. using the quantum tori variables $\{u_i,v_i\}$ (see \cite{GKLO} and references therein). However, the generators in their work do not seem to be represented by positive self-adjoint operators, and the transcendental relations characterizing the positive representations do not seem to hold. On the other hand, our representations of the lower Borel part $\cU_q^-$ generated by $F_i$, under the embedding in Theorem \ref{Thmqtori}, essentially coincide with what is called the Feigin map (see e.g. \cite{R}), thus the construction of the positive representations also extends the Feigin map to the whole quantum group. However, as in the work by Teschner \textit{et al.} on $\cU_{q\til[q]}(\sl(2,\R))$, it is clear from the construction that there is no classical $b\to 0$ limit, hence the class of positive representations distinguishes itself from being the usual $q$-deformation of classical representations.

\subsection{Organization of the paper}
The paper is organized as follows. In Section \ref{sec:prelim} we recall the definition of the quantum group $\cU_q(\g_\R)$ and Lusztig's parametrization for a positive unipotent semigroup $U_{>0}^+$, and the transformation between the coordinate corresponding to the change of words for $w_0$. In Section \ref{sec:def} we give the construction for $\cU_q(\g_\R)$ on a specific choice of reduced expression of $w_0$. In Section \ref{sec:qdilog} we recall the definition and properties of the quantum dilogarithm function needed in Section \ref{sec:trans}, where we prove Theorem \ref{Thmtrans} by defining the unitary transformation bringing together the action of $\cU_q(\g_\R)$ for any choice of expression for $w_0$. In Section \ref{sec:com} we prove the commutation relations between the generators for simply-laced type $\cU_q(\g_\R)$. In Sections \ref{sec:An} and \ref{sec:Dn} we write the action for $\cU_q(\sl(n,\R))$ and $\cU_q(\mathfrak{so}(n,n))$ explicitly in Lusztig coordinates. In Section \ref{sec:E} we give explicit results of the calculations for type $E_6$,$E_7$ and $E_8$. In Section \ref{sec:modified} we recall the modified quantum group $\bU_{\fq\til[\fq]}(\g_\R)$ defined in \cite{FI} and state the main theorems about the positive representations of the modular double, its embedding into the $q$-tori, and the Langlands dual inside the commutant. In Section \ref{sec:weylaction} we prove the unitary equivalence between positive representations with parameters related by Weyl actions. Finally in Section \ref{sec:conjectures} we give some remarks on the possible approaches to the conjectures stated in the introduction.

\vspace{10mm}

\textbf{Acknowledgments.} I would like to dedicate this work to my advisor Professor Igor Frenkel, who has enlightened me in this very beautiful area of mathematics, for all his support and guidance over the years at Yale University. I would like to thank both referees for their constructive comments to the manuscript in particular to the handling of unbounded operators. This work was partially supported by World Premier International Research Center Initiative (WPI Initiative), MEXT, Japan and JSPS KAKENHI
Grant Numbers JP16K17571.

%=======================================================================
\section{Definition of $\cU_q(\g_\R)$ and Lusztig data}\label{sec:prelim}
Throughout the paper we denote $\bi:=\sqrt{-1}$. Let $\g$ be a simple Lie algebra over $\C$ of simply-laced type. Let $I=\{1,2,...,r\}$ denote the set of nodes of the Dynkin diagram of $\g$ with $r=rank(\g)$. Let $W$ denote the Weyl group and $w_0\in W$ the longest element. We let $N=l(w_0)=dim(U_{>0}^+)$ (cf. Definition \ref{U+}), and we call a tuple
\Eq{\bii=(i_1,...,i_N)\in I^N}
a reduced word if $w_0=s_{i_1}s_{i_2}...s_{i_N}$ is a reduced expression of $w_0$. Finally we let $A=(a_{ij})$ be the Cartan matrix of $\g$, such that $i,j$ are connected in the Dynkin diagram whenever $a_{ij}=-1$.

\begin{Def}[\cite{D,J}]
The Drinfeld--Jimbo quantum group $\cU_q(\g)$ is the Hopf algebra generated by $\{E_i,F_i,K_i^{\pm1}\}_{i\in I}$ over $\C$ subject to the relations for $i,j\in I$,
\Eq{
K_iE_j=q^{a_{ij}}E_jK_i,\tab K_iF_j=q^{-a_{ij}}F_jK_i,\tab
{[E_i,F_j]} = \d_{ij}\frac{K_i-K_i\inv}{q-q\inv},
}
together with the Serre relations for $i\neq j$ and $X=E,F$,
\Eq{
\case{X_i^2X_j - (q+q\inv)X_iX_jX_i + X_jX_i^2=0,& a_{ij}=-1,\\ X_iX_j=X_jX_i&a_{ij}=0.}\label{SerreE}
}

The Hopf algebra structure of $\cU_q(\g)$ is given by (we will not need the counit and antipode in this paper)
\Eq{
\label{coprodE}\D(E_i)=&1\ox E_i+E_i\ox K_i,\\
\label{coprodF}\D(F_i)=&K_i^{-1}\ox F_i+F_i\ox 1,\\
\label{coprodK}\D(K_i)=&K_i\ox K_i.
}
The split real quantum group $\cU_q(\g_\R)$ is a real form of $\cU_q(\g)$, which is a Hopf-* algebra \cite{T} equipped in addition with a star structure defined by
\Eq{E_i^*=E_i,\tab F_i^*=F_i,\tab K_i^*=K_i.}
together with $q^*=q\inv$, which forces $|q|=1$. We will assume $q$ is not a root of unity.
\end{Def}

We recall the description of Lusztig data for the positive unipotent subgroup $U_{>0}^+$, given in detail in \cite{Lu}. 
\begin{Def}\label{U+}Let $G$ be the real simple Lie group corresponding to the split real form $\g_\R$ of the Lie algebra $\g$, such that it has a split real maximal torus $T$ and two opposite Borel subgroups $B^+, B^-$ containing $T$ with unipotent subgroups $U^+, U^-$. For any $i\in I$, there exists a homomorphism $SL_2(\R)\to G$ denoted by
\begin{eqnarray}
\veca{1&a\\0&1}&\mapsto& x_i(a)\in U_i^+,\\
\veca{b&0\\0&b\inv}&\mapsto &\chi_i(b)\in T,\\
\veca{1&0\\c&1}&\mapsto &y_i(c)\in U_i^-,
\end{eqnarray}
called the \emph{pinning} of $G$, where $U_i^+$ and $U_i^-$ are the simple root subgroups of the unipotent subgroups $U^+$ and $U^-$ respectively. Then the positive unipotent semigroup $U_{>0}^+$ is defined by the image of the map $\iota:\R_{>0}^N\to U^+$ given by
\Eq{\label{w0coord}\iota:(a_1,a_2,...,a_N)\mapsto x_{i_1}(a_1)x_{i_2}(a_2)...x_{i_N}(a_N).}
\end{Def}
\begin{Lem}[{\cite[Proposition 2.7]{Lu}}]\label{inj} The map $\iota:\R_{>0}^N\to U^+$ is injective: if 
$$x_{i_1}(a_1)x_{i_2}(a_2)...x_{i_N}(a_N)=x_{i_1}(a'_1)x_{i_2}(a'_2)...x_{i_N}(a'_N),$$
then $a_k=a'_k$ for every $k=1,...,N$.
\end{Lem}
\begin{Def} We define the totally positive semigroup 
\Eq{\label{gauss}
G_{>0}:=U_{>0}^- T_{>0} U_{>0}^+,
} where $U_{>0}^\pm$ is as above, and $T_{>0}$ are generated by the images $\chi_i(b)$ with $b\in\R_{>0}$.
\end{Def}
\begin{Lem} We have the following identities in $G_{>0}$: for $a,b,c\in \R_{>0}$ and $i,j\in I$,
\begin{eqnarray}
\label{EK}\chi_i(b)x_i(a)&=&x_i(b^2 a)\chi_i(b),\\
\label{EF}x_i(a)y_j(c)&=&y_j(c)x_i(a)\tab \mbox{ if $i\neq j$},\\
\label{EKF}x_i(a)\chi_i(b)y_i(c)&=&y_i(\frac{c}{ac+b^2})\chi_i(\frac{ac+b^2}{b})x_i(\frac{a}{ac+b^2}).
\end{eqnarray}

Assume $a_{ij}=-1$. Then we have
\begin{eqnarray}
\label{EK2}\chi_i(b)x_j(a)&=&x_j(b\inv a)\chi_i(b),\\
\label{121}x_i(a)x_j(b)x_i(c)&=&x_j(\frac{bc}{a+c})x_i(a+c)x_j(\frac{ab}{a+c}).
\end{eqnarray}
\end{Lem}

%=======================================================================
\section{Construction of the positive representations}\label{sec:def}
We will first construct the action of the generators $E_i$, $F_i$, $K_i$ of $\cU_q(\g_\R)$ on the Hilbert space $L^2(\R^N)$ for a particular $i\in I$ using a specific choice of reduced word $\bii$. Then using Theorem \ref{trans} we can find the actions corresponding to an arbitrary choice of $\bii$ by a unitary transformation, hence in particular the positivity and the transcendental relations are preserved once we prove it for some $\bii$. Finally, we will prove in Section \ref{sec:com} that the actions really satisfy the defining commutation relations of $\cU_q(\g_\R)$.

First let us recall the classical construction, which corresponds to the induced representation $\mathrm{Ind}_{B^-}^G$ of the lower Borel (i.e. minimal parabolic) subgroup $B^-$, acting on the ring of smooth functions with compact support $C_0^\oo(G/B^-)$ on the flag variety. (See e.g. \cite[VII.3]{Kn}. For our purpose we do not need to consider the $L^2$ completion). It is known that the open big Bruhat cell $G^{w_0,w_0}/B^-$ is isomorphic to the unipotent group $U_+$. Since we are only interested in the infinitesimal action, we can restrict our action to functions $C_0^\oo(U_{>0}^+)$ of the positive chart.
\begin{Prop}\label{classical} The minimal principal series representation for $\cU(\g)$ can be realized as the infinitesimal action of $g\in G_{>0}$ acting on $C_0^\oo(U_{>0}^+)$ by
\Eq{g\cdot f(h) =\chi_\l(hg) f([hg]_+).}
Here we write the Gauss decomposition \eqref{gauss} of $g$ as
\Eq{g=g_-g_0g_+\in U_{>0}^-T_{>0}U_{>0}^+,}
so that $[g]_+=g_+$ is the projection of $g$ onto $U_{>0}^+$, and $\chi_\l(g)$ is the character function defined by
\Eq{\chi_\l(g):=\prod_{i=1}^r t_i^{2\l_i},}
where $\l=(\l_i)\in \C^r$ and $t_i=\chi_i\inv(g_0)\in \R_{>0}$.
\end{Prop}

Let $q=e^{\pi \bi b^2}$ with $b^2\in\R\setminus\Q$, $0<b<1$, let
\Eq{
[n]_q:=\frac{q^n-q^{-n}}{q-q\inv}}
be the $q$-number and denote $Q:=b+b\inv$.
 
\begin{Def}\label{labelling}Let $\bii=(i_1,...,i_N)$ be a reduced word, and denote the Lusztig coordinates of $C_0^\oo(U_{>0}^+)\simeq C_0^\oo(\R_{>0}^N)$ by $\bx=(x_j)_{j=1}^N$ such that $x_j$ corresponds to the $j$-th letter $i_j$. 
 
Following the approach in \cite{FI}, we apply the formal Mellin transformation of the form
$$\cM: F(\bx)\mapsto f(\bu):=\int F(\bx)\bx^\bu d\bx$$
 on the functions $F(\bx)\in C_0^\oo(U_{>0}^+)$, which transforms differential operators in $\bx$ into finite difference operators in the Mellin-transformed variables $\bu=(u_j)_{j=1}^N$.
\end{Def}
Here, by \emph{differential operators in $\bx$} we mean linear combinations of operators of the form $\dis P(\bx)\left(\del[,\bx]\right)^{\ba}$ where $P(\bx)$ is a Laurent polynomial in the $\bx$ variables and $\dis\left(\del[,\bx]\right)^\ba:=\prod_{n=1}^N \left(\del[,x_j]\right)^{a_j}$ for some nonnegative integer vector $\ba=(a_1,...,a_N)\in\Z_{\geq0}^N$.

\begin{Rem} In the \emph{formal} Mellin transformation we do not care about the metric and can just take $d\bx$ to be the standard Lebesgue measure for $U_{>0}^+\simeq \R_{>0}^N$. We will see later that the resulting finite difference operators induced from the regular action of $U(\g)$ on $C_0^\oo(U_{>0}^+)$ extend to well-defined operators on the \emph{polynomial ring} $\C[u_1,...,u_N]$, and still form a representation of $\cU(\g)$, as proved in Propositions \ref{PropE}, \ref{PropF} and \ref{PropH} below for the action of its generators. 

As shown in \cite{Ip2}, using the Haar measure instead, one can modify $\cM$ to become the usual unitary transformation $L^2(U_{>0}^+,dg)\to L^2(\R^N, d\bu)$. However, this will just introduce an extra shift in $\bu$ by a constant in the transformed action, which can be absorbed into the parameter $\a$ in $D^q$ below (cf. \eqref{qact}), and is not important for our purpose.
\end{Rem}
\begin{Def} Let $\cW\subset L^2(\R^N)$ be the dense subspace of entire rapidly decreasing functions spanned by functions of the form
\Eq{
\label{entire}e^{-\bu A\bu^T+\bc\cdot\bu}p(\bu)
}
for any $A\in M_{N\x N}(\C)$ with $Re(A)$ positive definite, $\bc\in \C^N$ and $p(\bu)$ a polynomial in $\bu$.
\end{Def}
Note that this space is preserved under the Fourier transformation \cite{FG}, and the unbounded multiplication operators $\cE_\bc:=e^{\pi b \bc\cdot\bu}$ for $\bc\in\C^N$ defined on $\cW$ are essentially self-adjoint with self-adjoint extension on 
\Eq{\label{domain}Dom(\cE_\bc):=\{f(\bu)\in L^2(\R^N): e^{\pi b \bc\cdot\bu}f(\bu)\in L^2(\R^N)\}.}

Now we introduce the following quantization method.
\begin{Def}\label{posquan} Given a finite difference operator on $\C[u_1,...,u_N]$ of the form $$D:f(\bu)\mapsto (1+\a+P(\bu))f(\bu+\be),$$ in the classical action, where $\a\in\C$ is a constant, $P(\bu)$ is a linear function over $\Q$ in the coordinate vector $\bu$ and $\be\in \R^N$ is a constant vector, we define the corresponding positive quantized action by
\Eq{\label{qact}D^q:=\left[\frac{Q}{2b}+\frac{\bi}{b}\a-\frac{\bi}{b}P(\bu)\right]_qe^{-2\pi bp_\be}}
acting on $\cW\subset L^2(\R^N)$ as unbounded operator, where the momentum operator $p_{\be}$ is defined such that $e^{-2\pi bp_\be}$ acts as $f(\bu)\mapsto f(\bu+\bi b\be)$. 
\end{Def}

The rescaling by $-\frac{\bi}{b}$ is such that we recover the standard quantum plane variables $\{e^{2\pi bu}, e^{2\pi bp}\}$ in the expressions
\Eq{e^{2\pi b u}e^{2\pi b p}=q^2 e^{2\pi b p}e^{2\pi b u}}
where the relation holds in the sense of unbounded operators interpreted appropriately \cite{Hall, PT2}, namely as bounded unitary operators,
\Eq{\label{oneparameter}e^{2\pi \bi b s u}e^{2\pi \bi bt p} = e^{-2\pi \bi stb^2}e^{2\pi \bi b t p}e^{2\pi \bi b s u},\tab \forall s,t\in \R.}

Due to the twisting $\frac{Q}{2b}=\frac{1}{2}+\frac{1}{2b^2}$, the quantized action has no classical limit $b\to 0$. The introduction of the twist comes from the following observation\footnote{We thank the referee for pointing out an argument that avoids the use of the Baker-Campbell-Hausdorff formula, which requires more justification on the domains of the unbounded operators.}:
\begin{Prop}\label{symmetric}
Expression \eqref{qact} is positive symmetric on $\cW$ whenever 
$$[p_\be,P(\bu)]=\frac{1}{2\pi \bi}.$$
\end{Prop}
\begin{proof} For $A$ a $\Q$-linear combination of $u_i$ (and possibly a constant $\a$) and $B$ a $\Q$-linear combination of $p_i$ as above such that they are self-adjoint on $L^2(\R^N)$ and $[B,A]=\frac{1}{2\pi \bi}$, we have
\Eqn{
\left[\frac{Q}{2b}-\frac{\bi}{b}A\right]_qe^{-2\pi b B} &= \frac{q^{\frac{Q}{2b}-\frac{\bi}{b}A}-q^{-\frac{Q}{2b}+\frac{\bi}{b}A}}{q-q\inv}e^{-2\pi b B}\\
&=\frac{\bi}{q-q\inv}(q^{\frac12}e^{\pi b A}e^{-2\pi b B} +q^{-\frac12}e^{-\pi b A}e^{-2\pi b B}).
}
Note that $\frac{\bi}{q-q\inv}=\frac{1}{2\sin(\pi b^2)}>0$. 
By a change of variable, which is a unitary transformation on $L^2(\R^N)$, one can transform $A,B$ to the standard single variable position and momentum operator $u,p$ on, say, the first coordinate of $L^2(\R^N)$. Hence it suffices to consider the positivity of the expressions 
\Eq{q^{\pm\frac12}e^{\pm\pi b u}e^{-2\pi b p}.}

Since the exponentiated commutation relations \eqref{oneparameter} for $u$ and $p$ hold as unitary operators on $\cW$ (and hence on $L^2(\R^N)$), one can define $\pm u-2p$ to be the infinitesimal generator of the strongly continuous one-parameter group of unitary operators
\Eq{
U_{\pm}(t):=e^{\mp\frac{\pi \bi t^2}{2}}e^{\pm\pi \bi t u}e^{-2\pi \bi t p},\tab t\in\R
}
so that 
\Eq{e^{\pi \bi t(\pm u-2p)}=e^{\mp\frac{\pi \bi t^2}{2}}e^{\pm\pi \bi t u}e^{-2\pi \bi t p},\tab t\in\R.}

As $\cW$ is a subspace of the analytic vectors of $U(t)$, one can analytic continue the expression and put $t=-\bi b$ to obtain
\Eq{
\label{BCH}q^{\pm\frac12}e^{\pm\pi b u}e^{-2\pi b p}=e^{\pi b(\pm u-2p)}
}
as operators on $\cW$, and hence they are positive symmetric.
\end{proof}
\begin{Rem}\label{ext}
Although in \eqref{BCH} the two operators coincide on $\cW$, the natural domain as unbounded operator of the right-hand side is much larger. In the sequel we will take this as the definition of the extension of the product of two unbounded positive operators of the form \eqref{BCH}.
\end{Rem}

We saw in Proposition \ref{symmetric} that expression \eqref{qact} is an unbounded positive \emph{symmetric} operator on $\cW\subset L^2(\R^N)$. In Corollary \ref{corpos} we will show that the positive quantized actions given by the formal expressions below are indeed positive essentially self-adjoint operators on $L^2(\R^N)$ with domains obtained by pullbacks of $Dom(\cE)$ given in \eqref{domain} for some multiplication operator $\cE$. 

Let us now describe the results of applying the above procedure to the generators of the Lie algebra $\g$. 

\begin{Prop}\label{PropE}Fix $i\in I$ and choose a reduced word $\bii$ with $i_N=i$. Then the right multiplication on $U_{>0}^+$ by $\exp(tE_i)$ is simply given by $x_N\mapsto x_N+t$, hence differentiating the action, $E_i$ acts on $C_0^\oo(U_{>0}^+)$ by
\Eq{E_i:=\left.\frac{d}{dt}\right|_{t=0}\exp(tE_i)=\pd_{x_N}:=\del[,x_N],}
where $x_N$ is the rightmost coordinate of the parametrization of $C_0^\oo(U_{>0}^+)$. Under the Mellin transform the action is given on $f(\bu)$ by
\Eq{E_i:f(\bu)\mapsto (u_N+1)f(..., u_N+1),}
where we omit the other coordinates that are unchanged.
\end{Prop}

According to Definition \ref{posquan}, we make the following definition.
\begin{Def}\label{E}
Fix a reduced word $\bii$ with $i_N=i$. The positive quantized action for $E_i$ is defined by
\begin{eqnarray}
\pi_{\bii}(E_i):=E_i^q&=&\left[\frac{Q}{2b}-\frac{\bi}{b}u_N\right]_qe^{-2\pi bp_N}\\
&=&\label{EE}\frac{\bi}{q-q\inv}(e^{\pi b(u_N-2p_N)}+e^{\pi b(-u_N-2p_N)}).
\end{eqnarray}
\end{Def}

In the propositions and definitions below, we let $\l_i\in\R$, and fix an arbitrary reduced word $\bii=(i_1,...,i_N)$. 

\begin{Prop}\label{PropF}The action of $\exp(tF_i)$ on $C_0^\oo(U_{>0}^+)$ is given by
\Eq{\label{eF1}
\exp(tF_i):F(\bx)\mapsto D_1(\bx)^{2\l_i}F(\what[\bx]),
}
where $\l_i\in\C$ are complex parameters, $D_k(\bx)$ and $\what[\bx]:=(\what[x]_k)$ are given by
\begin{eqnarray}
D_k(\bx)&:=&1+t\sum_{\substack{k\leq j\leq N\\ i_j=i}}x_j,\label{eF2}\\
\what[x]_k&=&\case{\dis\frac{x_k}{D_k(\bx)D_{k+1}(\bx)}& a_{i,i_k}=2,\label{eF3}\\
x_k D_k(\bx)& a_{i,i_k}=-1,\\
x_k&a_{i,i_k}=0.}
\end{eqnarray}
The action of $F_i$ is given by
\Eq{
F_i:=\left.\frac{d}{dt}\right|_{t=0}\exp(tF_i):=\sum_{k:i_k=i}x_k\left(-\sum_{j=1}^{k-1} a_{i,i_j} x_j\pd_{x_j}+x_k\pd_{x_k}+2\l_i\right).\label{eF4}
}
The Mellin transformation of the above action is:
\Eq{\label{FFr}
F_i:f(\bu)\mapsto\sum_{k:i_k=i}\left(1-\sum_{j=1}^{k-1} a_{i,i_j}u_j-u_k+2\l_i\right)f(u_k-1).}
\end{Prop}

\begin{proof} Let us first prove \eqref{eF1}--\eqref{eF3}. Since the pinning $x_i(t)$ for the upper unipotent part has the same symbol with the Lusztig coordinates $(x_k)$, let us change the name of the coordinates to $\ba:=(a_k)$ instead. Using \eqref{EF}--\eqref{EKF} and by induction, the right multiplication of $y_i(t)=\exp(tF_i)$ on $g\in U_{>0}^+$ can be rewritten as
\Eqn{
g\cdot\exp(tF_i)&=\left(\prod_{k=1}^N x_{i_k}(a_{i_k})\right)y_i(t)=y_i(t)\left(\prod_{k=1}^N \chi_{i_k}\left(\frac{D_k(\ba)}{D_{k+1}(\ba)}\right)x_{i_k}\left(a_{i_k}\frac{D_{k+1}(\ba)}{D_k(\ba)}\right)\right),
}
where $D_k(\ba)$ is given by \eqref{eF2} with $x_k$ replaced by $a_k$. Note that $\frac{D_k(\ba)}{D_{k+1}(\ba)}=1$ if $i_k\neq i$, hence we can ignore the $\chi_{i_k}$ terms for $i_k\neq i$. Now using \eqref{EK} and \eqref{EK2}, we move all the $\chi_i$ to the left and obtain again by induction:
\Eqn{
g\cdot\exp(tF_i)=y_i(t)\left(\prod_{k:i_k=i}\chi_i\left(\frac{D_k(\ba)}{D_{k+1}(\ba)}\right)\right)\left(\prod_{k=1}^N x_{i_k}(\what[a]_k)\right)=y_i(t)\chi_i(D_1(\ba))\left(\prod_{k=1}^N x_{i_k}(\what[a]_k)\right),
}
where $\what[a]_k$ is given by \eqref{eF3} with $x_k$ replaced by $a_k$. This gives the Gauss decomposition of $g\cdot\exp(tF_i)$. Therefore by Proposition \ref{classical}, we quotient out $y_i(t)$ and raise the character $\chi_i$ to the power $2\l_i$ and obtain \eqref{eF1}.

To get \eqref{eF4} we differentiate the action using the chain rule to get
\Eq{
F_i &= -\sum_{k:i_k=i} (x_k+2\sum_{\substack{k+1\leq j\leq N\\i_j=i}} x_j)x_k \pd_{x_k}+\sum_{k: a_{i,i_k}=-1}(\sum_{\substack{k\leq j\leq N\\i_j=i}}x_j)x_k\pd_{x_k}+2\l_i\sum_{k:i_k=i} x_k.
}
Upon comparing the terms of the form $x_jx_k\pd_{x_k}$, we can rearrange the summation and obtain the required form \eqref{eF4}. Finally \eqref{FFr} can be obtained from \eqref{eF4} by the standard procedure described in \cite{FI}.
\end{proof}

Again according to Definition \ref{posquan}, we make the following definition.
\begin{Def}\label{F}
The positive quantized action for $F_i$ is given by:
\Eq{\label{FF}
\pi_{\bii}(F_i):=F_i^q&=\sum_{k: i_k=i}\left[\frac{Q}{2b}+\frac{\bi}{b}\left(\sum_{j=1}^{k-1}a_{i,i_j} u_j+u_k+2\l_i\right)\right]_q e^{2\pi b p_k}\\
&=\frac{\bi}{q-q\inv}\sum_{k:i_k=i}\left(e^{\pi b\left(\sum_{j=1}^{k-1}a_{i,i_j} u_j+u_k+2\l_i\right)+2\pi bp_k}+e^{-\pi b\left(\sum_{j=1}^{k-1}a_{i,i_j} u_j+u_k+2\l_i\right)+2\pi bp_k}\right).\nonumber
}
\end{Def}

Finally the following is easy to obtain by applying \eqref{EK} and \eqref{EK2} repeatedly:
\begin{Prop} \label{PropH}The action of $exp(tH_i)$ on $C_0^\oo(U_{>0}^+)$ is given by 
\Eq{exp(tH_i):F(\bx)\mapsto e^{2\l_i t}f(..., e^{-a_{i,i_j}t}x_j, ...),}
so that
\Eq{H_i:=\left.\dd[,t]\right|_{t=0}\exp(tH_i)= \sum_{j=1}^{N} -a_{i,i_j}x_j\pd_{x_j}+2\l_i.}
The Mellin-transformed action is simply multiplication by
\Eq{\label{HH}H_i= \sum_{j=1}^{N} -a_{i,i_j}u_j+2\l_i,}
\end{Prop}
\begin{Def}\label{H}
The positive quantized action for $K_i:=q^{H_i}$  (with the corresponding rescaling by $-\frac{\bi}{b}$) is given by
\Eq{\label{KK}\pi_{\bii}(K_i):=K_i^q=e^{-\pi b\left(\sum_{j=1}^{N} a_{i,i_j}u_j+2\l_i\right)}.}
\end{Def}

In the rest of the paper, by abuse of notation we will just write $E_i$ for $\pi_{\bii}(E_i)$ when it is clear from the context that we are talking about representations on some Hilbert space with the choice of $\bii$ being understood.
%=======================================================================
\section{Quantum dilogarithm}\label{sec:qdilog}
Let us briefly recall the definition of the quantum dilogarithm functions, central to the calculations in the next section. This function $G_b(x)$ and its variant $g_b(x)$ were first introduced in \cite{FKa} and motivate the definition of the modular double \cite{Fa}. We summarize those of properties that are needed in the next section. References can be found e.g. in \cite{Ip} and \cite{PT2}. 

Again let $q=e^{\pi \bi b^2}$ and $Q=b+b\inv$, where $b^2\in \R\setminus \Q$ and $0<b<1$.

\begin{Def} The quantum dilogarithm function $G_b(z)$ is defined on\\ ${0\leq Re(z)\leq Q}$ by
\Eq{\label{intform} G_b(z):=\over[\ze_b]\exp\left(-\int_{\W}\frac{e^{\pi tz}}{(e^{\pi bt}-1)(e^{\pi b\inv t}-1)}\frac{dt}{t}\right),}
where \Eq{\ze_b=e^{\frac{\pi \bi}{2}(\frac{b^2+b^{-2}}{6}+\frac{1}{2})},}
and the contour goes along $\R$ with a small semicircle going above the pole at $t=0$. This can be extended meromorphically to the whole complex plane.
\end{Def}
\begin{Def} The function $g_b(z)$ is defined by
\Eq{g_b(z):=\frac{\over[\ze_b]}{G_b(\frac{Q}{2}+\frac{\log z}{2\pi \bi b})}.}
where $\log$ takes the principal branch of $z$.
\end{Def}
We will need the following properties:
\begin{Lem}\label{gbunit} We have 
\Eq{G_b(z)G_b(Q-z)=e^{\pi \bi z(z-Q)}.}
Furthermore $|g_b(z)|=1$ when $z\in\R_{>0}$, hence in particular $g_b(X)$ is a unitary operator for any positive operator $X$.
\end{Lem}
\begin{Lem}[$q$-binomial Theorem]\label{qbinom}
For positive self-adjoint variables $u,v$ with ${uv=q^2vu}$,
\Eq{(u+v)^{\bi b\inv t}=\int_{C}\veca{\bi t\\\bi\t}_b u^{\bi b\inv (t-\t)}v^{\bi b\inv \t}d\t ,}
where the $q$-beta function (or $q$-binomial coefficient) is defined by
\Eq{\veca{t\\\t}_b:=\frac{G_b(-\t)G_b(\t-t)}{G_b(-t)},}
and $C$ is the contour along $\R$ that goes above the pole at $\t=0$
and below the pole at $\t=t$.
\end{Lem}
\begin{Lem}[tau-beta Theorem]\label{tau} We have
\Eq{\int_C e^{-2\pi \t \b}
\frac{G_b(\a+\bi \t)}{G_b(Q+\bi \t)}d\t =\frac{G_b(\a)G_b(\b)}{G_b(\a+\b)},}
where the contour $C$ goes along $\R$ and goes above the poles of
$G_b(Q+\bi\t)$ and below those of $G_b(\a+\bi\t)$. By the asymptotic properties of $G_b$, the integral converges for ${Re(\b)>0}$, ${Re(\a+\b)<Q}$.
\end{Lem}
\begin{Lem}\label{qsum}For positive self-adjoint variables $u,v$ with $uv=q^2vu$ we have
\Eq{\label{qsum1}g_b(u)^*vg_b(u) = q\inv uv+v,}
\Eq{\label{qsum2}g_b(v)ug_b(v)^*=u+q\inv uv.}
\end{Lem}

%=======================================================================
\section{Transformations of the representations}\label{sec:trans}
In this section we derive the transformations of the actions corresponding to the change of reduced word $\bii$ of the longest element $w_0$. First let us consider the quantum analogue of the Lusztig decomposition and its transformation.

\begin{Prop} Let $\C[q,q\inv]\<\a,\b,\c\>$ be the noncommutative field of fractions generated by a triplet of quantum variables $(\a,\b,\c)$ satisfying
\Eq{
\a\b=q^2\b\a, \tab\c\a=q^2\a\c, \tab\b\c=\c\b} and let $\{\a',\b',\c'\}\in \C[q,q\inv]\<\a,\b,\c\>$ be defined by

\begin{eqnarray}\label{a'b'c'}
\a'&=&(\a+\c)\inv \c\b=\b\c(\a+\c)\inv,\\
\b'&=&\a+\c,\nonumber\\
\c'&=&(\a+\c)\inv \a\b=\b\a(\a+\c)\inv.\nonumber
\end{eqnarray}
such that they satisfy the commutative relations described by the diagrams
\Eq{\label{flip}\begin{tikzpicture}[baseline={([yshift=3ex]c.base)}]
\node (a) at (0, 0) {$\a$};
\node (b) at +(50: 1.17) {$\b$};
\node (c) at +(0: 1.5) {$\c$};
\foreach \from/\to in {c/a, a/b}
\draw [-angle 90] (\from) -- (\to);
\end{tikzpicture}
\tab\mbox{and}\tab
\begin{tikzpicture}[baseline={([yshift=3ex]b.base)}]
\node (a) at (0, 0) {$\a'$};
\node (b) at +(-50: 1.17) {$\b'$};
\node (c) at +(0: 1.5) {$\c'$};
\foreach \from/\to in {b/c, c/a}
\draw [-angle 90] (\from) -- (\to);
\end{tikzpicture}}
where $\a\to \b$ means that $\a,\b$ satisfy the relation $\a\b=q^2\b\a$.

Then a quantum analogue of the Lusztig transformation holds, where
\Eq{
x_2(\a)x_1(\b)x_2(\c)=x_1(\a')x_2(\b')x_1(\c')
} 
in the sense of the embedding of a $3\x 3$ matrix with quantum variable entries, and the map
\Eq{\phi:(\a,\b,\c)\mapsto(\a',\b',\c')}
is an involution between the quantum variables $(\a,\b,\c)\corr(\a',\b',\c')$.
\end{Prop}

\begin{Rem}
The pair satisfying \eqref{flip} forms the quantum cluster chart of the \emph{quantum unipotent part $U_q^+$} of the positive quantum group $GL_q^+(3,\R)$ defined in \cite{Ip2}, where we established the so-called \emph{Gauss--Lusztig decomposition} $GL_q^+(n,\R)=U_q^-T_qU_q^+$.
\end{Rem}
The commutation relation \eqref{flip} is not symmetric between $(\a,\b,\c)$ and $(\a',\b',\c')$. To give a universal description and extend the map $\phi$ to a decomposition of general type, we define the quantum variables as follows.
\begin{Def}
We assign an orientation
$$\cdots\circ_i \to \circ_j\cdots$$ for each pair of connected nodes of the Dynkin diagram, and given a reduced word $\bii$, we define the quantum variables $(\a_k)_{k=1}^N$ with $q$-commutation relations given by
\begin{eqnarray}\label{assign}
\a_m\longleftarrow \a_n&\case{i_m=i_n & m<n,\\i_m=i, i_n=j, &m>n,}
\end{eqnarray}
whenever $\circ_i\to\circ_j$ in the Dynkin diagram, or the corresponding variables commute otherwise if $i,j$ are not adjacent in the Dynkin diagram.
\end{Def}
For example, the relations in \eqref{flip} for $(\a,\b,\c)$ and $(\a',\b',\c')$ correspond to the orientation $\circ_1\to\circ_2$ of the Dynkin diagram of type $A_2$ with the reduced word $\bii=(2,1,2)$ and $\bii=(1,2,1)$ respectively.
In particular, if $m<k<n$, with $i_m=i_n=i$ and $i_k=j$, the triplet of quantum variables $(\a_m,\a_k,\a_n)$ satisfies \eqref{flip} in one of the orientations.

\begin{Prop}\label{phipreserve}Define the transformation of quantum variables corresponding to the change of words $\bii=(...,i,j,i,...)\corr \bii'=(...,j,i,j,...)$  at position $p>1$ by
\Eq{\phi:\a_k\mapsto \a_k',\tab k=1,...,N,}
where $\a_k':=\a_k$ for $k\notin\{p-1,p,p+1\}$ and the triplet
$(\a_{p-1}',\a_{p}',\a_{p+1}')$ is defined by \eqref{a'b'c'}. Then $\phi$ preserves the rule of assignments \eqref{assign} of $q$-commutation relations of $(\a_k')$ for the word $\bii'$ under the same orientation of the Dynkin diagram.
\end{Prop}
Furthermore, for an appropriate choice of orientation of the Dynkin diagram corresponding to the reduced expression of $w_0$ given by the arrows between frozen nodes of the basic quiver \cite{Ip8}, the quantum Lusztig variables are known to be monomials of the quantum cluster $\mathcal{X}$ variables \cite{FG}; hence an analogue of Lemma \ref{inj} still holds in the quantum case \cite{BZ} through the nonvanishing of the generalized quantum minors ($\cA$ variables), so that the above transformations of variables are consistent with different choices of change of reduced words $\bii$.\footnote{We thank L. Shen for the discussion on the consistency of the change of words. Recently a more geometric argument using DT-invariant has been presented in \cite{GS}.}

Let us now specialize to the case where $(\a,\b,\c)$ are \emph{positive}, i.e. by letting $q\in\C$ with $|q|=1$ and let $\a=\a^*$ be expressed in terms of positive essentially self-adjoint operators on some Hilbert space $\cH$. By Lemma \ref{qsum}, the sum $\a+\c$ is unitarily equivalent to a positive operator, hence $(\a',\b',\c')$ can be extended to positive operators as well (cf. Remark \ref{ext}). Using the idea of \cite{Ip, PT1}, let us consider the $C^*$-algebra of operators $\cA$ generated by the quantum coordinates $\a,\b,\c$ as in \eqref{flip} viewed as positive operators on $\cH$. We let $f(u,v,w)$ be a rapidly decreasing function in $L^2(\R^3)$ and consider the expression of ``symmetric type"
\Eq{
\int_{\R^3} f(u,v,w)\a^{-\bi b\inv u/2}\b^{-\bi b\inv v}\c^{-\bi b\inv w}\a^{-\bi b\inv u/2}dudvdw.
}
Since $\a,\b,\c$ are positive, the expression makes sense as a \emph{bounded operator} on $\cH$ by functional calculus and $\cA\subset \cB(\cH)$. Here symmetric type means that the order of the quantum variables will not be changed when taking the complex conjugate (since $\b$ and $\c$ commute), hence the order is canonical. In other words complex conjugation gives the involution $f(u,v,w)\mapsto f(-u,-v,-w)$. Note that we previously rescaled the variables by $\bi b$, hence the $-\bi b\inv$ factor appears in the definition. 

Then the involution \eqref{flip} induces a unitary transformation on $f(u,v,w)$ as follows.
\begin{Prop}\label{quanPhi} The Mellin transform $\cM:\cA\to L^2(\R^3)$ of the map $\phi$ induces a unitary transformation $\Phi$ on $L^2(\R^3)$ of the coordinate functions on the symmetric type integral
\begin{eqnarray*}
\cM \phi&:&\iiint f(u,v,w)\a^{-\bi b\inv u/2}\b^{-\bi b\inv v}\c^{-\bi b\inv w}\a^{-\bi b\inv u/2}dudvdw \\
&\longmapsto &\iiint (\Phi f)(u,v,w){\c'}^{-\bi b\inv w/2}{\a'}^{-\bi b\inv u}{\b'}^{-\bi b\inv v}{\c'}^{\bi b\inv u/2}dudvdw,
\end{eqnarray*}
such that $\Phi f\in L^2(\R^3)$ with $\Phi$ a unitary operator given by
\Eq{\Phi = T\circ M,}
where 
\Eq{M=g_b(e^{\pi b(2p_w-2p_u+u-v+w)})\circ g_b^*(e^{\pi b(2p_w-2p_u-u+v-w)}),}
and $T$ is the composition of unitary transformations:
\Eq{T=(u\corr v)\circ (v\corr w)\circ (u\mapsto u-w)\circ (v\mapsto v+w),}
or simply the transformation matrix of determinant 1:
\Eq{T\cdot \veca{u\\v\\w}=\veca{-1&1&0\\1&0&1\\1&0&0}\veca{u\\v\\w}.}
In particular, $\phi$ is an involution implies $\Phi:=\cM\phi\cM\inv$ is also an involution.
\end{Prop}

\begin{Rem}
Note that by Lemma \ref{gbunit}, $M$ is a unitary operator expressed in terms of the functional calculus applied to the quantum dilogarithm function $g_b$ and its complex conjugate. It can be written explicitly as an integral operator using the Fourier transformation formula \cite{Ip}
\Eq{
g_b(X)=\int_{\R+\bi0}\frac{e^{-\pi\bi t^2}}{G_b(Q+\bi t)}X^{\bi b\inv t}dt.
}
\end{Rem}
\begin{proof}
After substituting the expression \eqref{a'b'c'} for $(\a',\b',\c')$ on the right-hand side of $\cM\phi$, we can use the $q$-binomial formula (Lemma \ref{qbinom}) to expand
$$(\a+\c)^{\bi b\inv(u-v+w)}=\int_{C} \c^{\bi b\inv(u-v+w-\t)}\a^{\bi b\inv\t}\frac{G_b(-\bi\t)G_b(-\bi u+\bi v-\bi w+\bi\t)}{G_b(-\bi u+\bi v-\bi w)}d\t.$$
After grouping the terms with corresponding $q$-factors, and with some changes of variables, we arrive at the transformation on the coordinates
$$f(u,v,w)\mapsto \int_{C}f(v-u-\t,u+w,u+\t)e^{\pi \bi\t(\t+u-v+w)}\frac{G_b(-\bi\t)G_b(\bi\t+\bi u-\bi v+\bi w)}{G_b(\bi u-\bi v+\bi w)}d\t,$$
which can be rewritten using Lemma \ref{tau} as 
\begin{eqnarray*}
\Phi&=&\int_C e^{\pi  \bi \t(2p_u-2p_w+\t+u-v+w)}\frac{G_b(-\bi \t)G_b(\bi \t+\bi u-\bi v+\bi w)}{G_b(\bi u-\bi v+\bi w)}d\t\circ T\\
&=&\frac{G_b(\frac{Q}{2}-\bi p_u+\bi p_w-\frac{1}{2}(\bi u-\bi v+\bi w))G_b(\bi u-\bi v+\bi w)}{G_b(\frac{Q}{2}-\bi p_u+\bi p_w+\frac{1}{2}(\bi u-\bi v+\bi w))G_b(\bi u-\bi v+\bi w)}\circ T\\
&=&\frac{g_b(e^{\pi  b(2p_u-2p_w-u+v-w)})}{g_b(e^{\pi  b(2p_u-2p_w+u-v+w)})}\circ T\\
&=&T\circ \frac{g_b(e^{\pi b(2p_w-2p_u+u-v+w)})}{g_b(e^{\pi b(2p_w-2p_u-u+v-w)})}.
\end{eqnarray*}
Finally one proceeds by similar calculations to show that the map going the other way is exactly the same. The techniques of calculations above involving the Heisenberg-type variables $p$, $u$ can be found e.g. in \cite{Ip}.
\end{proof}

\begin{Thm}\label{trans} 
Let $(...,j,k,...)\corr (...,k,j,...)$ be a change of reduced words with $a_{jk}=0$, and the corresponding coordinates are given by $(u,v)$. Then the change of coordinates is simply given by $u\corr v$.

On the other hand, let $(...,j,k,j,...)\corr (...,k,j,k,...)$ be a change of reduced words with $a_{jk}=-1$. We relate the action of $E_i$ by the unitary transformation
\Eq{E_i\mapsto \Phi E_i\Phi\inv,}
where $\Phi$ is given in Proposition \ref{quanPhi} for the corresponding coordinates.

Then the actions of $F_i$ and $K_i$ constructed in Definitions \ref{F} and \ref{H} are also related in the same way as
\Eq{F_i\mapsto \Phi F_i\Phi\inv,\tab K_i\mapsto \Phi K_i\Phi\inv.}
\end{Thm}

\begin{proof}
The first statement is trivial.

For the second case, the transformation of $E_i$ is consistent by the remarks after Proposition \ref{phipreserve}, while to check that $\Phi$ correctly maps $F_i$ and $H_i$, it suffices to check their transformed action directly using Lemma \ref{qsum} on the coordinate system of the form
$$x_i(\a)x_j(\b)x_i(\c)\corr x_j(\a')x_i(\b')x_j(\c')$$ and $$x_k(\a)x_j(\b)x_k(\c)\corr x_j(\a')x_k(\b')x_j(\c'),$$ where $a_{ij}=a_{jk}=-1$ and $a_{ik}=0$.

\end{proof}

This proves Theorem \ref{Thmtrans}. In particular, we can now state the definition of the \emph{positive representations} in the Main Theorem of the introduction.

\begin{Def} Fix a reduced word $\bii$. The \emph{positive representations} $\cP_\l$ parametrized by $\l\in\R^{r}$ are the representations of $\cU_q(\g_\R)$ defined as follows. For each $i\in I$, the action of the generator $E_i$ is defined to be $\Phi E_i'\Phi\inv$, where $\Phi$ is a transformation by Theorem \ref{trans} corresponding to the change of words sending $\bii'$ to $\bii$ with $i_N'=i$ and $E_i'$ is the positive quantized action defined in Definition \ref{E} for $\bii'$. The actions of the generators $F_i$, $K_i$ are simply defined by the positive quantized action in Definitions \ref{F}, \ref{H} respectively. Furthermore, by Theorem \ref{trans}, positive representations corresponding to different choices of reduced words are unitarily equivalent.
\end{Def}

\begin{Cor}\label{corpos} The positive representations of the generators of $\cU_q(\g_\R)$ are realized by positive essentially self-adjoint operators on a dense domain $\cD\subset \cH:=L^2(\R^N)$, and they satisfy the transcendental relations \eqref{transcendenal}.
\end{Cor}
\begin{proof} Since $K_i=q^{H_i}$ is an exponential operator, up to a change of variables it is unitarily equivalent to $\cE:=e^{\pi b x_1}$, the multiplication operator in the first variable on $\cH$, which is essentially self-adjoint with the self-adjoint domain given by
\Eq{
Dom(\cE)=\{f\in \cH : e^{\pi b x_1}f\in \cH\},
} and the transcendental relations also follow from $(e^{\pi b x_1})^{\frac{1}{b^2}}=e^{\pi b\inv x_1}$ on $Dom(e^{\pi b\inv x_1})$.

The proof for $E_i$ and $F_i$ also follows from the argument in \cite{FI}. Recall that $E_i$ is defined using the simple expression in \eqref{EE}. In particular, both $E_i$ and $F_i$ are unitarily equivalent to an expression of the form
\Eq{\label{EFform}\frac{\bi}{q-q\inv}\sum_k (A_k^++A_k^-),}
where $A_k^\pm$ are $q^2$-commuting terms, in the exact same way as \cite[(3.33) in Theorem 3.4]{FI}. Hence they are further unitarily equivalent to $\cE$ by conjugating with a sequence of quantum dilogarithms and changes of variables successively.

In other words, for each generator $X=E_i,F_i$ of $\cU_q(\g_\R)$, there exists a unitary operator $\Phi_{X}:\cH\to \cH$ such that $Ad_{\Phi_{X}}(X)=\cE$, hence both $E_i$ and $F_i$ are essentially self-adjoint on the pullback of the domain of $\cE$ by $\Phi_X$ and the transcendental relations also follow as above.

Finally, we take $\cD$ to be the intersection of all the domains of polynomial expressions of the generators, defined by the pullback of the domain of the corresponding exponential operators using $\Phi_X$. $\cD$ is dense since it contains the entire functions from $\cW$ (cf. \eqref{entire}).
\end{proof}
In particular, for $\cU_q(\sl(2,\R))$ we recover the domain described by \cite{PT2}, which is obtained from conjugation of the domains of $e^{\pi b x}$ and $e^{\pi b p}$ by a single quantum dilogarithm, and the intersection of all the domains gives the analytic property of the domain of representation.

It is clear from the definition of the coproduct \eqref{coprodE}-\eqref{coprodK} that the two summands $q^2$-commute with each other, hence we also have the following corollary.
\begin{Cor}The coproduct of the generators acting on $L^2(\R^{2N})$ is also positive essentially self-adjoint, and satisfies the transcendental relations.
\end{Cor}

Finally we can recover the representations of the usual universal enveloping algebra by considering a quasi-classical limit:
\begin{Cor}\label{classreplimit}
We can read off the action of the classical $\cU(\g_\R)$ from the quantized action, where we consider identification of the form
$$\left[\frac{Q}{2b}+\frac{\bi }{b}P(\bu)\right]_qe^{2\pi b\bp} \to f(\bu)\mapsto\left(\frac{1}{2}+\bi P(\bu)\right)f(\bu-\bi \be).$$
Hence this gives the classical principal representations of $\cU(\g_\R)$ for any choice of reduced expression of $w_0$ as finite difference operators.
\end{Cor}

More precisely, using Theorem \ref{transLi} in Section \ref{sec:conjectures}, we observe that under the shifting of the variables $u_k$ by $\frac{\bi Q}{2}\b_k$, each weight factor $\left[\frac{Q}{2b}+\frac{\bi}{b}P(\bu)\right]_q$ becomes the standard quantum number $[P(\bu)]_q$ and does not depend on $\frac{1}{b}$. In particular since the action of each $H_i$ under this shifting changes only by a constant, any invariant subspace will go under the classical limit to an invariant subspace of the classical principal series representations. Since we know that the classical principal series representations are irreducible, we have the following conclusion.
\begin{Prop} The positive representations $\cP_\l$ of $\cU_q(\g_\R)$ are irreducible for all $\l_i\in\R$.
\end{Prop}

\begin{Rem}\label{highestweight}
If we do a ``Wick rotation" of the variables $\bu\mapsto -\bi b \bu$ and choose the parameters $\l_i$ discretely as $\l_i\in \frac{Q}{2}-ib\N$, then we can also recover the finite-dimensional highest weight representations of $\cU_q(\g)$ as a quotient of $\cP_\l$. Hence one can think of the positive representations as a kind of ``analytic continuation" of the finite-dimensional representations of the compact quantum groups. This observation is utilized e.g. in \cite{Ip7} in the study of positive Casimir operators.
\end{Rem}
\begin{Rem} In \cite{Wo} a theory of $C^*$-algebras generated by unbounded operators was developed in full generality. In this paper, we are treating only positive operators, so the structure is significantly simpler since we do not need to deal with the phases of the unbounded operators, and therefore can apply functional calculus to raise them to complex powers in order to produce the bounded operators needed for the computations.
\end{Rem}
%=======================================================================
\section{Proof of commutation relations}\label{sec:com}
In order to prove the commutation relations, we just need to pick a reduced word $\bii$ such that the action is simple enough for us to check the relations, since by Theorem \ref{trans} the commutation relations are preserved under conjugation of a unitary transformation $\Phi$ associated with different choice of reduced words. Again let $N=l(w_0)$ be the length of $\bii$.

\begin{Prop} We have $[E_i,F_i]=\frac{K_i-K_i\inv}{q-q\inv}=\frac{q^{H_i}-q^{-H_i}}{q-q\inv}$.
\end{Prop}
\begin{proof} Choose $\bii$ such that $i_N=i$. Consider the coordinates $u_j$ as in Definition \ref{labelling}. Then 
$$E_i=\left[\frac{Q}{2b}-\frac{\bi}{b}u_N\right]_qe^{-2\pi bp_N},$$ while the only term in $F_i$ not commuting with $E_i$ is the last term:
$$F_i=\left[\frac{Q}{2b}+\frac{\bi}{b}\left(\sum_{j=1}^{N-1}a_{i,i_j}u_j+u_N+2\l_i\right)\right]_qe^{2\pi bp_N}+\mbox{(commuting with $E_i$)},$$
and $H_i$ is of the form $$\frac{\bi}{b}\left(\sum_{j=1}^{N-1} a_{i,i_j}u_j+2u_N+2\l_i\right).$$
Then the commutation relation follows directly from the identity
\Eq{[A]_q[B-1]_q-[A-1]_q[B]_q=[B-A]_q.}
\end{proof}

\begin{Prop} We have $[E_i,F_j]=0$ if $i\neq j$.
\end{Prop}
\begin{proof} Again choose $\bii$ such that $i_N=i$. Then it is obvious since $F_j$ does not involve any shifting of $p_k$ in the variables $u_k$ with $i_k=i$.
\end{proof}

\begin{Prop} We have $K_iE_j=q^{a_{ij}}E_jK_i$ and $K_iF_j=q^{-a_{ij}}F_jK_i$.
\end{Prop}
\begin{proof} For $E_j$, choose $\bii$ such that $i_N=j$. Then it is immediate from the expression of the action $E_j$ in \eqref{EE} and $K_i$ in \eqref{KK}. For $F_j$ it follows easily from the expression in \eqref{FF} since the only momentum operator appearing in the expression is $p_k$ with $i_k=j$, while $K_i$ has the exponential of $a_{i,j}u_k$.
\end{proof}

\begin{Prop} We have $[E_i,E_j]=[F_i,F_j]=0$ if $a_{ij}=0$.
\end{Prop}
\begin{proof} The case for $E_i$ follows by choosing the reduced word of the form $\bii=(...,i,j)$. (i.e. the reduced expression of $w_0$ is of the form $w_0=w_{N-2}s_is_j$, where $w_{N-2}:=w_0s_js_i\in W$ is necessarily of length $N-2$). The case for $F_i$ follows immediately from its expression in Definition \ref{F}, since the shiftings of the variables do not appear in one another.
\end{proof}
\begin{Prop} We have the Serre relations for $E_i$.
\end{Prop}
\begin{proof} Let $a_{ij}=-1$. We choose the reduced word of the form $\bii=(...,i,j,i)$. (i.e. the reduced expression of $w_0$ is of the form $w_{N-3}s_is_js_i\in W$ where $w_{N-3}:=w_0s_is_js_i$ is necessarily of length $N-3$).  Now the action of $E_i$ and $E_j$ only depends on the first three variables, hence it follows immediately from the corresponding reduced expression for $w_0$ in the type-$A_2$ case which is easy to check.
\end{proof}

\begin{Prop} We have the Serre relations for $F_i$.
\end{Prop}
\begin{proof}
Using the expression in Definition \ref{F}, let us rewrite the action of $F_i$ as
\Eq{F_i\cdot f(\bu)=\sum_{k:i_k=i}F_i^k(\bu) f(\bu-\bi b\be_k)=\sum_{k:i_k=i}F_i^k(\bu)e^{2\pi bp_k}f(\bu),}
where $\be_k$ denotes the standard vector with entry $1$ in the coordinate $u_k$ and 0 elsewhere. Then the Serre relations are equivalent to the vanishing of the commutation factor $B_q(a,b,c)$ which is exactly the same as \cite[(3.10) in proof of Theorem 3.1]{FI} with the same variables $a,b,c$ defined by comparing the shifts in $\be_k$ with $F_i^{k'}(\bu)$ for various $k$ and $k'$. Hence the Serre relations for $F_i$ follow.
\end{proof}

Therefore we have constructed the positive representations satisfying properties (i) and (ii) in the Main Theorem of the introduction. In order for it to be compatible with the modular double, we will give the definition of the modified quantum group $\bU_{\fq\til[\fq]}(\g_\R)$ in Section \ref{sec:modified}. But first let us write down explicitly the positive representations for quantum groups of type $A_n$, $D_n$ and discuss some calculations involving $E_n$.

%=======================================================================
\section{Explicit expression of positive representations}\label{sec:explicit}
In this section we give the explicit expression of positive representations for type $A_n$, $D_n$ and also comment on type $E_n$, where the explicit expression can be found in the Appendix.

First we introduce some notation:
\begin{Def}\label{labelling2}Let $(\bu)=(u_j)_{j=1}^N$ be the Mellin-transformed Lusztig coordinates defined in Definition \ref{labelling} corresponding to a reduced word $\bii=(i_1,...,i_N)$. We define a new indexing $(u_i^k)$ as
$$u_i^k:=u_j$$
where $i:=i_j,$ and $$k:=\#\{l:j\leq l, \mbox{ and } i_j=i_l\}.$$
In other words, $u_i^k$ is the $k$-th variable from the right of $\bii$ having root index $i$. We also denote the corresponding momentum operator by $p_i^k$.
\end{Def}
\begin{Ex} The coordinates for type $A_3$ corresponding to $\bii=(3,2,1,3,2,3)$ is given by
$$(u_3^3,u_2^2,u_1^1,u_3^2,u_2^1,u_3^1):=(u_1,u_2,u_3,u_4,u_5,u_6)$$
\end{Ex}
\subsection{Positive representations for type $A_n$}\label{sec:An}

In \cite{FI} we studied the positive representations of $\cU_q(\sl(n,\R))$ for the standard expression of $w_0$ using the ``cluster coordinates" of $U_{>0}^+$. These are the initial minors of $U_{>0}^+\subset SL^+(n+1,\R)$, which are the determinants of the square submatrices that start from the top row. Let $X_{i,j}$ denote the initial minors of the square submatrix with lower-right entry at $(i,j)$. We choose the standard reduced word
\Eq{
\bii=(n,n-1,...,3,2,1,n,n-1,...,3,2,n,n-1,...,3,...,n)
}
where we label the Dynkin diagram by
\begin{center}
  \begin{tikzpicture}[scale=.4]
    \draw[xshift=0 cm,thick] (0 cm, 0) circle (.3 cm);
    \foreach \x in {1,...,5}
    \draw[xshift=\x cm,thick] (\x cm,0) circle (.3cm);
    \draw[dotted,thick] (8.3 cm,0) -- +(1.4 cm,0);
    \foreach \y in {0.15,...,3.15}
    \draw[xshift=\y cm,thick] (\y cm,0) -- +(1.4 cm,0);
    \foreach \z in {1,...,5}
    \node at (2*\z-2,-1) {$\z$};
\node at (10,-1){$n$};
  \end{tikzpicture}
\end{center}

We have the following relations between $X_{i,j}$ and the Lusztig data $x_i^k$ (corresponding to the Mellin-transformed variable $u_i^k$ using the notation from Definition \ref{labelling2}).

\begin{Prop} The cluster coordinates $X_{i,j}$ and the Lusztig data $x_i^k$ are related by
\Eq{x_{i}^{j}=\frac{X_{j,i+1}X_{j-1,i-1}}{X_{j,i}X_{j-1,i}},}
\Eq{X_{i,i+j}=\prod_{l=1}^j \prod_{k=1}^i x_{k+l-1}^{k}.}
Here we denote $X_{i,i}=X_{i,0}=X_{0,j}=1$.
\end{Prop}

Since these become just linear transformations in the Mellin-transformed variables, we can write down explicitly the action of \cite{FI} in terms of the Lusztig coordinates.

To simplify notation, let us introduce
\begin{Def}\label{up} We denote
\Eq{[u]:=\left[\frac{Q}{2b}-\frac{\bi}{b}u\right]_q,\tab e(p):=e^{2\pi bp}.}
Then if $[p,u]=\frac{1}{2\pi \bi}$ we  have
\Eq{[u]e(-p)=\left[\frac{Q}{2b}-\frac{\bi}{b}u\right]_qe^{-2\pi bp}=\left(\frac{\bi}{q-q\inv}\right)(e^{\pi b(u-2p)}+e^{\pi b(-u-2p)}).}
which is positive self-adjoint. Note that $[u]e(-p)=e(-p)[-u]$.
\end{Def}

\begin{Thm} The actions of $E_i$, $F_i$ and $K_i$ are given by
\begin{eqnarray*}
E_i&=&\sum_{k=1}^{n-i+1} [u_{i+k-1}^{k}-u_{i+k}^{k}]e\left(\sum_{l=1}^k (p_{i+l-1}^{l-1}-p_{i+l-1}^{l})\right),\\
F_i&=&\sum_{k=1}^i\left[u_{i}^{k}-\sum_{l=k}^i(2u_{i}^{l}-u_{i-1}^{l}-u_{i+1}^{l+1})+2\l_i\right]e(p_{i}^{k}),\\
K_i&=&e^{\pi b(\sum_{k=1}^{i}(u_{i-1}^{k}+u_{i+1}^{k}-2u_{i}^{k})-2\l_i)},
\end{eqnarray*}
where $u_{i}^{k}=p_{i}^{k}=0$ if the indices are out of bound.
\end{Thm}

%=======================================================================

\subsection{Positive representations for type $D_n$}\label{sec:Dn}
For each choice of reduced words $\bii$ for $w_0$, we have a positive representation. In general, however, it seems that there exists the ``best" choice of expression so that the representation is ``minimal" in the sense that the total number of terms is minimized. We conjecture that this can be written in the form:
\Eq{w_0=w_{t_1}w_{t_2}...w_{t_n}}
so that $w^k:=w_{t_1}...w_{t_k}$ gives the longest element of the group corresponding to a connected Dynkin subdiagram with nodes $t_1,...,t_k$, with $l(w^k)$ as long as possible.
\begin{Rem} The longest word $w_0$ chosen here and the next subsection coincides with the decomposition of $w_0$ used in \cite{Li} which is called ``nice" (except for type $E_8$), where the sequence $(t_1,..., t_n)$ above forms a ``good enumeration" associated to a sequence of certain descending Levi subalgebras. It was shown that the longest word compatible with this sequence will be unique up to a change of reflections $s_is_j=s_js_i$ which corresponds to a change of variable indices in the positive representations. We thank the referee for pointing out this reference.
\end{Rem}

Let the labeling of the Dynkin diagram for $D_n$ be
\begin{center}
  \begin{tikzpicture}[scale=.4]
    \draw[xshift=0 cm,thick] (0 cm, 1) circle (.3 cm);
    \draw[xshift=0 cm,thick] (0 cm, -1) circle (.3 cm);
    \foreach \x in {1,...,5}
    \draw[xshift=\x cm,thick] (\x cm,0) circle (.3cm);
    \draw[dotted,thick] (8.3 cm,0) -- +(1.4 cm,0);
   \draw[xshift=0.25 cm] (0 cm,1) -- +(1.4 cm,-1);
   \draw[xshift=0.25 cm] (0 cm,-1) -- +(1.4 cm,1);   
 \foreach \y in {1.15,...,3.15}
    \draw[xshift=\y cm,thick] (\y cm,0) -- +(1.4 cm,0);
    \foreach \z in {2,...,5}
    \node at (2*\z-2,-1) {$\z$};
\node at (10,-1){$n-1$};
\node at (-1,-1){$0$};
\node at (-1,1){$1$};
  \end{tikzpicture}
\end{center}

For type $D_n$, the corresponding sequence is 
$$(t_1,...,t_n)=(2,1,0,3,4,5,6,7,...,n-1)$$
Then $w_0$ of length $n(n-1)$ corresponds to the reduced word
\Eq{\label{w0simplify}\bii=(2\;12\;012\;\;320123\;\;43201234\;\;...\;\;(n-1)...43201234...(n-1)) }
where we omit the commas in $\bii$ for typesetting purposes.

However, we will transform $(212012...)$ at the beginning of $\bii$ to $(012012...)$ so that the expression can be made symmetric in $E_0$ and $E_1$, but with slightly more terms. This will become useful in the study of the non-simply-laced case \cite{Ip4} due to the \emph{folding} operation.

The explicit formula for the positive representations can actually be carried out involving only the simple transformation rule
\Eq{\label{rule}\Phi:[w]e(-p_w)\mapsto [u]e(-p_u-p_v+p_w)+[v-w]e(-p_v)}
using Proposition \ref{quanPhi}, and we obtain the following general formula by a series of inductions.

\begin{Thm}\label{ThmDn} The positive representations for $\cU_q(\g_\R)$ of type $D_n$, parametrized by $\l\in \R^r$, are, for $i=0$ or $1$,
\begin{eqnarray*}
E_i&=&\sum_{k=1}^{n-1}[u_{\over[k+i-1]}^{k}-u_{2}^{2k-1}] e\left(\sum_{l_0=1}^{s_1(k)}(-1)^{l_0}p_{i}^{l_0}-\sum_{l_1=1}^{s_2(k)}(-1)^{l_1}p_{1-i}^{l_1}-\sum_{l_2=1}^{2k-2}(-1)^{l_2}p_{2}^{l_2}\right)\\
&&+\sum_{k=1}^{n-2}[u_{2}^{2k}-u_{\over[k+i]}^{k}]e\left(\sum_{l_0=1}^{s_1(k)}(-1)^{l_0}p_{i}^{l_0}-\sum_{l_1=1}^{s_2(k)}(-1)^{l_1}p_{1-i}^{l_1}-\sum_{l_2=1}^{2k}(-1)^{l_2}p_{2}^{l_2}\right)
\end{eqnarray*}
and for $i\geq2$,
$$E_i=\sum_{k=1}^{2n-2i-1} [(-1)^k(u_{i+1}^{k}-u_{i}^{k})]e\left(\sum_{l_0=1}^{s_1(k)}(-1)^{l_0}p_{i}^{l_0}-\sum_{l_1=1}^{s_2(k)}(-1)^{l_1} p_{i+1}^{l_1} \right),$$
where $\over[k]:=k \mbox{ (mod 2)}\in\{0,1\}$, and
\begin{eqnarray*}
s_1(k):=2\left\lceil\frac{k}{2}\right\rceil  -1&,&s_2(k):=2\left\lfloor\frac{k}{2}\right\rfloor.
\end{eqnarray*}

Note that the dependence of $\l_i$ is hidden in the actions of $F_i$ and $K_i$, which are given as before by Definitions \ref{F} and \ref{H}, and as before we ignore the variables where the indices are out of bound (see Table 2 in the next section for the restrictions of $m$ for $u_{i}^{m}$).
\end{Thm}

%=======================================================================

\subsection{Positive representations for type $E_n$}\label{sec:E}
We follow the same strategy as in type $D_n$, by choosing a ``good" expression for $w_0$ coming from adding nodes successively to the Dynkin diagram: $$w_0=w_{t_1}w_{t_2}...w_{t_n}.$$
It turns out that the best case is obtained by the embedding
$$D_4\sub D_5\sub E_6\sub E_7\sub E_8,$$
where we label the Dynkin diagram by the following index
\begin{center}
  \begin{tikzpicture}[scale=.4]
    \draw[xshift=0 cm,thick] (0 cm, 0) circle (.3 cm);
    \foreach \x in {1,...,6}
    \draw[xshift=\x cm,thick] (\x cm,0) circle (.3cm);
    \foreach \y in {0.15,...,5.15}
    \draw[xshift=\y cm,thick] (\y cm,0) -- +(1.4 cm,0);
    \foreach \z in {1,...,7}
    \node at (2*\z-2,1) {$\z$};
\draw[xshift=0 cm,thick] (4 cm, -2) circle (.3 cm);
  \draw[xshift=0 cm] (4 cm,-0.25) -- +(0 cm,-1.5);
\node at (4,-3){$0$.};
  \end{tikzpicture}
\end{center}
The reduced words $\bii$ used, together with the sequence $(t_1,t_2,...,t_n)$ are
\begin{eqnarray*}
\bii(E_6)&=&(4\;34\;034\;230432\;12340321\;5432103243054321),\\
(t_1,..,t_6)&=&(4,3,0,2,1,5);\\
\bii(E_7)&=&(4\;34\;034\;230432\;12340321\;5432103243054321\;654320345612345034230123456),\\
(t_1,..,t_7)&=&(4,3,0,2,1,5,6);\\
\bii(E_8)&=&(4\;34\;034\;230432\;12340321\;5432103243054321\;654320345612345034230123456\\
&&765432103243546503423012345676543203456123450342301234567),\\
(t_1,..,t_8)&=&(4,3,0,2,1,5,6,7).
\end{eqnarray*}

The explicit actions of $E_i$ are given in the Appendix, which is reproduced from \cite{Ip3}. Here we summarize the numerical results in the following theorem.
\begin{Thm}\label{NumTerms} We have constructed the positive representation for a simply-laced quantum group, where the number of terms of the form $[\bu]e(\bp)$ in the actions of $E_i$ and $F_i$ are summarized in Tabel 1.

{\footnotesize
$$\begin{array}{c|c|c|c|c|c|}
generator&A_n&D_n&E_6&E_7&E_8\\
\hline
E_0&&2n-3&9&15&27\\
E_1&n&2n-3&1&11&23\\
E_2&n-1&2n-5&11&13&25\\
E_3&n-2&2n-7&10&16&28\\
E_4&\vdots&2n-9&7&17&29\\
E_5&\vdots&2n-11&5&7&19\\
E_6&\vdots&\vdots&&1&23\\
E_7&\vdots&\vdots&&&1\\
\vdots&\vdots&\vdots&&&\\
E_k&n-k+1&2n-2k-1&&&\\
\vdots&\vdots&\vdots&&&\\
E_{n-1}&2&1&&&\\
E_{n}&1&&&&\\
\hline
Total:&\frac{n(n+1)}{2}&n^2-2&43&80&175
\end{array}$$}
\begin{center}Table 1. Number of terms for the action of the generators $E_i$\end{center}

The number of terms in $F_i$ corresponds to the number of $i$ appearing in $\bii$, see Table 2.
{\footnotesize
$$\begin{array}{c|c|c|c|c|c|}
generator&A_n&D_n&E_6&E_7&E_8\\
\hline
F_0&&n-1&5&8&14\\
F_1&1&n-1&4&6&10\\
F_2&2&2n-4&7&11&19\\
F_3&3&2n-6&10&16&28\\
F_4&\vdots&2n-8&8&13&23\\
F_5&\vdots&2n-10&2&6&14\\
F_6&\vdots&\vdots&&3&9\\
F_7&\vdots&\vdots&&&3\\
\vdots&\vdots&\vdots&&&\\
F_k&k&2n-2k&&&\\
\vdots&\vdots&\vdots&&&\\
F_{n-1}&n-1&2&&&\\
F_{n}&n&&&&\\
\hline
Total:&\frac{n(n+1)}{2}&n(n-1)&36&63&120
\end{array}$$}
\begin{center}Table 2. Number of terms for the action of the generators $F_i$\end{center}
\end{Thm}

Using Corollary \ref{classreplimit}, these explicit expressions of positive representations for $\cU_q(\g_\R)$ give explicitly the classical principal series representations for $\cU(\g_\R)$.

\begin{Rem} In the above choice for $\bii$, only the simple transformation rule \eqref{rule} is involved. In general, however, things can get complicated. In particular we can have, for example,
$$\Phi:[w-u]e(p_v-p_w)\mapsto[v-2w]e(p_u-p_w)+[2]_q[u-w]e(p_w-p_v)+[2u-v]e(2p_w-p_u-p_v).$$
This even happens, for example, in the case of type $A_3$ with $\bii=(132132)$. Therefore we see that the quantization rule $(2u)\mapsto [2u]$ holds only in the initial setup described previously, and it is not always true in general.
\end{Rem}
\begin{Rem}
The number of terms can get really large if the choice for the reduced word of $\bii$ is not optimal, due to the transformation rule which produces new terms exponentially. For example, in type $D_n$ or $E_n$, if we choose $w_0$ of the form $$w_0=w'w_{A_{n-2}}s_0,$$ where $w_{A_{n-2}}$ is the standard subword of type $A_{n-2}$, so that the actions of the generators are all very simple except the action of $E_i$ corresponding to the branch point of the Dynkin diagram, then compared with Table 1, we obtain e.g. 1043 terms for $E_6$, 77565 terms for $E_7$, and over 1 million terms for $E_8$ for the action of the generator $E_3$.
\end{Rem}
%=======================================================================
\section{Modified quantum group}\label{sec:modified}
Finally, recall from \cite{FI} that the generators $\{E_i, F_i, K_i\}$ and $\{\til[E_i],\til[F_i],\til[K_i]\}$ of the two parts of the modular double commute only up to a sign. Following the approach in \cite{FI}, we modify the generators with powers of $K_i$ in order to take care of the commutation relations. The definition is slightly modified to fit subsequent work on non-simply-laced group \cite{Ip4}.

Let us consider a bipartite version of the Dynkin diagram.
\begin{Def}\label{modified} For each node $i\in I$ in the Dynkin diagram, we assign a weight $n_i\in\{0,1\}$ such that $|n_i-n_j|=1$ if $i,j$ are connected in the diagram, so that $n_i$ alternates along the edges.

We define $\fq:=q^2=e^{2\pi \bi b^2}$ and
\Eq{\fq_i:=\case{\fq\inv&\mbox{if $n_i=0$,}\\ \fq&\mbox{if $n_i=1$,}}} and define the modified quantum generators as
\begin{eqnarray*}
\bE_i&:=&q^{n_i} E_iK_i^{n_i},\\
\bF_i&:=&q^{1-n_i}F_iK_i^{n_i-1},\\
\bK_i&:=&\fq_i^{H_i}=\case{K_i^{-2}&\mbox{if $n_i=0$,}\\K_i^{2}&\mbox{if $n_i=1$.}}
\end{eqnarray*}
\end{Def}
\begin{Prop}Let 
\Eq{[A,B]_\fq=AB-\fq\inv BA}
be the quantum commutator. Then the quantum relations in the new variables become
\begin{eqnarray}
\label{bKE}\bK_i\bE_j&=&{\fq_i}^{a_{ij}}\bE_j\bK_i,\\
\label{bKF}\bK_i\bF_j&=&{\fq_i}^{-a_{ij}}\bF_j\bK_i,\\
\bE_i\bF_j&=&\bF_j\bE_i\tab \mbox{ if $i\neq j$},\\
{[\bE_i,\bF_i]}_{\fq_i}&=&\frac{1-\bK_i}{1-\fq_i},
\end{eqnarray}
 and the quantum Serre relations become
\Eq{\label{bSerre}[[\bE_j,\bE_{i}]_{\fq_i},\bE_i]=0=[[\bF_j,\bF_{i}]_{\fq_i},\bF_i].}
\end{Prop}
\begin{Def}
We define the modified quantum group $\bU_\fq(\g)$ to be the Hopf algebra generated by $\bE_i$, $\bF_i$, $\bK_i$ over $\C[\fq^\pm]$ subject to the relations \eqref{bKE}-\eqref{bSerre}.

Let $\fq=e^{2\pi i b^2}$. Then we define the split real $\bU_{\fq}(\g_\R)$ to be the Hopf-$*$ algebra where the generators are self-adjoint under the star structure with $\fq^*=\fq\inv\in\C$.

Similar to the unmodified case, let $\til[\fq]:=\til[q]^2=e^{2\pi \bi b^{-2}}$. We define $\bU_{\til[\fq]}(\g_\R)$ by representing the generators $\til[\bE]_i$, $\til[\bF]_i$, $\til[\bK]_i$ using the formulas above with all the terms replaced by tilde, and we let $\bU_{\fq\til[\fq]}(\g_\R):=\bU_{\fq}(\g_\R)\ox \bU_{\til[\fq]}(\g_\R)$ denote the modular double.
\end{Def}

We can now state the Main Theorem from the introduction for $\bU_{\fq\til[\fq]}(\g_\R)$

\begin{Thm} We define the rescaled generators $\be_i$, $\bf_i$, $\bK_i$ by the same rescaling \eqref{rescale} as before. Then 
the positive representations of $\cU_{q\til[q]}(\g_\R)$ induce a positive representation on $\bU_{\fq\til[\fq]}(\g_\R)$, where we denote the operators under the representation with the same symbol as the generators:
\begin{itemize}
\item[(i)] the operators $\be_i$, $\bf_i$, $\bK_i$ and their tilde counterparts are represented by positive essentially self-adjoint operators;
\item[(ii)] we have the transcendental relations
\begin{eqnarray}
(\be_i)^{\frac{1}{b^2}}&=&\til[\be_i],\\
(\bf_i)^{\;\frac{1}{b^2}}&=&\til[\bf_i],\\
(\bK_i)^{\frac{1}{b^2}}&=&\til[\bK_i];
\end{eqnarray}
\item[(iii)] \mbox{all the generators $\{\bE_i,\bF_i,\bK_i\}$ commute weakly with $\{\til[\bE]_j,\til[\bF]_j,\til[\bK]_j\}$.}
\end{itemize}
Therefore we have constructed the positive principal series representations of the modular double $\bU_{\fq\til[\fq]}(\g_\R)$ where $\g$ is of simply-laced type, parametrized by $rank(\g)$ numbers $\l_i\in\R$.
\end{Thm}

Let $\bC:=\C[\fq^{\pm1}, \L_i^{\pm1}]_{i\in I}$ be the commutative ring of coefficients, and let $\T_{\fq}^{N}:=\bC\<\bu_k^{\pm1},\bv_k^{\pm1}\>_{k=1}^N$ be the Laurent polynomials in the quantum tori variables that satisfy
\Eq{\bu_{k}\bv_{k}=\fq \bv_{k}\bu_{k},}
and commute otherwise.

Then we have Theorem \ref{Thmqtori} generalizing the results of type $A_n$ in \cite{FI}:

\begin{Thm} We have an embedding
\Eq{\bU_{\fq}(\g)\inj \T_{\fq}^{N},}
where $N=l(w_0)=\dim(U_{>0}^+)$. We recover the positive representations $\cP_\l$ by letting $\fq=e^{2\pi i b^2}$, and representing the quantum tori variables by positive self-adjoint operators:
\Eq{\label{fulltorus}\bu_{k}=e^{2\pi bu_{k}},\tab \bv_{k}=e^{2\pi bp_{k}},\tab \L_i=e^{\pi b \l_i}}
for the usual position $u_k$ and momentum operators $p_k=\frac{1}{2\pi \bi}\del[,u_k]$ acting on $L^2(\R^N)$. This can be extended to the positive representations $\cP_\l$ of the modular double by the embedding
$$\bU_\fq(\g)\ox \bU_{\til[\fq]}(\g)\inj \T_{\fq}^N\ox \T_{\til[\fq]}^N$$
acting on the same space $L^2(\R^N)$.
\end{Thm}
\begin{Rem} In \cite{SS}, such a realization is called a \emph{polarization} of the positive representation.
\end{Rem}
\begin{proof} By the explicit expressions for the operators of $\bE_i$ and $\bF_i$ constructed in the previous sections (modified by the $K_i$ factors), we note that all operators are sums of terms that $q^2$-commute with each other. Therefore there exists a unitary transformation such that we can diagonalize the symplectic form corresponding to these $q^2$-commuting terms and obtain a realization in terms of the standard tori. Explicitly, it can be obtained by substituting $e^{\pi b \l_i}$ by $\L_i$, followed by multiplications by the unitary functions
\Eqn{
e^{\pi \bi u_ku_l}&:2p_k\mapsto 2p_k+u_l,\;\;\;p_l\mapsto 2p_l+u_k,& k<l,\\
e^{\frac{1}{2}\pi \bi{(u_k)}^2}&:2p_k\mapsto 2p_k+u_k,
}
whenever $n_{i_k}=0$ and $a_{i_ki_l}=-1$. 

By Remark \ref{highestweight}, the Wick rotation of the parameters $\l$ recovers every irreducible finite-dimensional representation of $\cU_q(\g)$, so that the embedding of every subset of the Poincar\'{e}--Birkhoff--Witt (PBW) basis \cite{Lu1} with a bounded degree will be faithful for some evaluations of $\L_i$. In particular, there will be no nontrivial polynomial relations among the generators of $\bU_{q}(\g)$ in the image, hence the map is an embedding.
\end{proof}

Finally we also relate the commutant of the representations with the Langlands dual, as described in Theorem \ref{ThmLL}. Recall that the Langlands dual is obtained by switching root lattice with coroot lattice. More explicitly, let
$\bb_i=\sum_{j\in I} b_j^i \a_j$ be the weight so that it is dual to the root $\a_i$:
$$(\bb_i,\a_j)=\d_{ij}.$$
This means that
\begin{eqnarray*}
&&\sum_{j\in I} b_j^i(\a_j,\a_k)=\d_{ik}\\
&\iff& \sum_{j\in I} b_{j}^i a_{kj}=\d_{ik}\\
&\iff& A\bb_i=\be_i
\end{eqnarray*}
where $A=(a_{ij})_{i,j\in I}$ is the Cartan matrix and $\be_i$ is the standard unit vector.

Let $D=\det A$ and consider an enlarged embedding
$$\iota:\bU_\fq(\g)\ox \bU_{\til[\fq]}(\g)\inj \T_\fq^N\ox \T_{\til[\fq]}^N\inj \what[\T]_{\fq\til[\fq], D}^N$$
where
$$\what[\T]_{\fq\til[\fq],D}^N:=\bC\<\bu_k^\frac1D,\bv_k^\frac1D, \til[\bu]_k^\frac1D,\til[\bv]_k^\frac1D\>_{k=1}^N$$
with the extra relations 
$$\bu_k^\frac1D\til[\bv]_k^\frac1D=\ze_{D^2} \til[\bv]_k^\frac1D\bu_k^\frac1D, \tab \bv_k^\frac1D\til[\bu]_k^\frac1D=\ze_{D^2}\inv \til[\bu]_k^\frac1D\bv_k^\frac1D$$
and $\ze_n$ is the primitive $n$-th root of unity.

\begin{Thm}\label{Langlands} Under the enlarged embedding, the elements $\what[\bK]_i\in \what[\T]_{\fq\til[\fq],D}^N$ for $i\in I$ commute with the image of $\bU_\fq(\g)$, where
\Eq{\label{KKKK}\what[\bK]_i:=\prod_{j\in I} \iota(\til[\bK]_j)^{b_j^i},}
and the vector $\bb_i=(b_j^i)_{j\in I}$ satisfies
\Eq{A\bb_i=\be_i}
\end{Thm}
In particular, the embedding $\bU_{\til[\fq]}(\g)\inj \T_{\til[\fq]}^N$ induces an embedding of the Langlands dual quantum group (i.e. the \emph{simply connected form}) $\bU_{\til[\fq]}(\g)\subset \bU_{\til[\fq]}({}^L\g)$ into $\what[\T]_{\fq\til[\fq],D}^N$, and the positive representations extended to the generators of $\bU_{\til[\fq]}({}^L\g_\R)$ (i.e. a polarization of $\what[\T]_{\fq\til[\fq],D}^N$) are operators weakly commuting with the generators of $\bU_\fq(\g_\R)$.

\begin{proof} For the $\til[\bK]_i$ generators, it suffices to choose a reduced expression $w_0$ for each $i$ and calculate the commutant with $\bE_i$. Let us choose $w_0=w_{l-1}s_i$. Then 
$$\left[\prod_{j\in I} \iota(\til[\bK]_j)^{b_j},\bE_i\right]=0$$
implies
$$\sum_{j\in I} a_{ij}b_j=(-1)^{n_i+1}k_i$$
for some integer $k_i$.
Combining for every $i$, we obtain the condition
$$A\bb = (k_i)_{i\in I}$$
for integers $k_i\in\Z$, hence the statement is proved.
\end{proof}

%=======================================================================
\section{Unitary equivalence $\cP_{\l}\simeq \cP_{w(\l)}$}\label{sec:weylaction}
So far we have constructed the positive principal series representations for the parameter $\l\in\R^r$ where $r$ is the rank of $\g$. We know that in the compact case, the finite-dimensional representations are parametrized by the cone of the dominant weights $P^+\subset \fh_\R^*$, where $\fh_\R$ is the real form of the Cartan subalgebra $\fh\subset\g$. Below we will show Theorem \ref{RW} that the positive representation $\cP_\l$ of $\cU_q(\g_\R)$ also depends only on the parameter $\l$ lying on the real span of the dominant cone.

\begin{Thm}Let $\cP_\l$ denote the positive principal series representations of $\cU_q(\g_\R)$ corresponding to the parameter $\l=(\l_i)_{i=1}^r$ where $r=rank(\g)$. Then 
\Eq{\cP_\l\simeq \cP_{w(\l)}}
are unitarily equivalent representations for any Weyl group element $w\in W$ acting on $\l$, namely for simple reflections,
\Eq{s_i(\l_j):=\l_j-a_{ij}\l_i=\case{-\l_i& i=j,\\\l_i+\l_j &\mbox{$i$ and $j$ are connected,}\\\l_j&\mbox{otherwise.}}}
In particular, the positive principal series representations are parametrized by $\l\in\R_{\geq 0}^r$.
\end{Thm}
\begin{proof} Let us first fix the root index $i$. Since representations corresponding to different reduced expressions of $w_0$ amount to unitary transformation by $\Phi$, we can take $w_0=s_iw_{l-1}$, and call $u:=v_1$ the leftmost variable corresponding to $s_i$. Then by Definition \ref{F} only the action of $F_i$ contains the term with $p_u$, namely, the term 
$$F_i=[-2\l_i-u]e(p_u)+(\mbox{independent of $p_u$})...$$
On the other hand, by the transformation rule \eqref{rule}, we see that the weight of the ``leftmost" term $[u]$ is always preserved. Since we know from the explicit expressions constructed in the previous sections that there is a unique $E_k$ with a single term involving the leftmost variable, this means that under the transformations $\Phi$ between different reduced expressions, there is a unique $E_k$ with the term of the form
$$E_k=[u]e(-p_u+...)+(\mbox{independent of $u, p_u$}).....$$
and no other terms from the action of the $E$ contain the variable $u$.
Now we can define our intertwiner $B$ as
\Eq{B=(u\mapsto u-\l_i)\circ G_{\l_i}(u) \circ (u\mapsto u-\l_i),}
where 
\Eq{G_{\l_i}(u)=\frac{g_b(e^{2\pi b(u+\l_i)})}{g_b(e^{2\pi b(u-\l_i)})}e^{-2\pi \bi\l_i u}}
is a unitary function which is essentially the same in the $\cU_{q\til[q]}(\sl(2,\R))$ case considered in \cite{PT2}. Note that if the action does not involve $p_u$, then it is just shifting by $u\mapsto u-2\l_i$. One can check that this map preserves the $E_k$ action, change the terms in $F_i$ as
$$[-2\l_i-u]e(p_u)\mapsto [2\l_i-u]e(p_u),$$
$$[-2\l_i-2u+...-w]e(p_w)\mapsto [2\l_i-2u+....-w]e(p_w),$$
 and for the action of $F_j$ with $j$ adjacent to $i$, all the terms change as
$$[-2\l_j+u+...-v]e(p_v)\mapsto [-2(\l_j+\l_i)+u+...-v]e(p_v).$$
\end{proof}
%=======================================================================

\section{Remarks on Conjecture \ref{conL2}}\label{sec:conjectures}
Finally we would like to discuss possible approaches to Conjecture \ref{conL2}. First let us generalize the result in \cite[Theorem 3.3]{FI}, where there exists a unitary transformation on $L^2(\R^N)$ so that the actions of $H_i$ do not depend on $\l_i$. In particular, this provides us with a way to find the Haar functional in order to define an $L^2$ space structure for the harmonic analysis of $L^2(G_q^+)$.

\begin{Thm}\label{transLi} There exists a linear transformation on the coordinates $u_j$ so that the action of $H_i$ does not depend on $\l_i$, while all the weights of $E_i$ and $F_i$ are of the form
$$\left[\frac{Q}{2b}+\frac{\bi }{b}\left(\sum c_ju_j+\l'_i\right)\right]_q,$$
for $r$ distinguished parameters $\l'_i$, where $r$ is the rank of $\g_\R$.
\end{Thm}
\begin{proof}
Let $\bii=(i_1,...,i_N)$ be the reduced word. We see that for each $k=1,...,N$, there exists a unique weight from some $F_i$ of the form
$$\left[\frac{Q}{2b}+\frac{\bi}{b}\left(\sum_{j=1}^k c_j u_j+2\l_{i_k}\right)\right]_q,$$
where $c_j\in\Z$ are some constants with $c_k=1$. Given a linear combination $\l=\sum_j d_j\l_j$ of the parameters $\l_i$, we define $\b(\l)=\sum_j d_j$. Then the transformation is given by the following.
\begin{itemize}
\item In the first step, we let $u_1\mapsto u_1-\l_{i_1}$.
\item In the $k$-th step, let the $k$-th weight be of the form 
$$\left[\frac{Q}{2b}+\frac{\bi}{b}\left(\sum_{j=1}^k c_j u_j+\l'_{i_k}\right)\right]_q$$
for some modified $\l'_{i_k}$. Denote $\b_k=\b(\l'_{i_k})-1$. Then we let $$u_k\mapsto u_k-\frac{\b_k}{\b_k+1}\l'_{i_k}.$$
\end{itemize}
Notice that $\b_k$ is always a positive integer. 
\end{proof}

This result allows us to state Conjecture \ref{conL2} in a more precise setting, namely to define an $L^2$ structure on (the modular double of) the quantized function space $F_{q\til[q]}(G^+)$. If we ignore the factor $\frac{\bi}{b}$ and treat the original the parameters $\l_i$ having real part $-\frac{Q}{2}$, then the shifting in the theorem above tells us that $-\frac{Q}{2}\b_k$ was the correct real part of $u_k$ to make the representation positive, which reflects its original classical Haar measure, namely the Haar measure on the original variables $a_k$ given by $a_k^{\b_k-1}da_k$. Then one can apply the method in \cite{Ip2} to define the space $L^2(G_q^+)$ in general using the Gauss--Lusztig decomposition. By studying the Plancherel measure arising from the decomposition of $L^2(G_q^+)$, it may shed some light on Conjecture \ref{contensor}, where in the case of $\cU_{q\til[q]}(\sl(2,\R))$ it is found in \cite{PT2} that the same measure shows up in the decomposition of the tensor product.

\begin{appendices}
\section {Positive representations for $E_6, E_7$ and $E_8$}
In this appendix we will write down explicitly the action of the operators $E_i$ in type $E_6$, $E_7$ and $E_8$. The actions of the operators $F_i$ and $K_i$ are given explicitly by Definitions \ref{F} and \ref{H} respectively. In \cite{Ip8} these representations are represented in a very clear graphical way on certain quivers using the language of quantum cluster algebra.

To simplify notation, let us use the following convention:
Let $P=\sum \pm p_{i}^k$ be a sum of the variables $p_i^k$, which denotes the shift in the variables $u_i^k$. Then we denote by $P(...p_m^n)$ the partial sum of $P$ starting from the first term and ending at $p_m^n$. For example, if 
$$P:=-p_{0}^{1}+p_{3}^{1}-p_{3}^{2}+p_{4}^{1}-p_{4}^{2}+p_{5}^{1}-p_{5}^{2}+p_{4}^{3}-p_{4}^{4}+p_{3}^{5}-p_{3}^{6}+p_{0}^{3}-p_{0}^{4}+p_{3}^{7}-p_{3}^{8}+p_{4}^{5}-p_{4}^{6},$$
then 
$$P(...p_4^4):=-p_{0}^{1}+p_{3}^{1}-p_{3}^{2}+p_{4}^{1}-p_{4}^{2}+p_{5}^{1}-p_{5}^{2}+p_{4}^{3}-p_{4}^{4}.$$
Also recall from Definition \ref{up} that
$$[\bu]e(-\bp):=\left[\frac{Q}{2b}-\frac{\bi}{b}(\bu)\right]_q e^{-2\pi b\bp}$$
and it is positive essentially self-adjoint when $[\bp,\bu]=\frac{1}{2\pi \bi}$.

%====================================================================================

\subsection{Positive representations in type $E_6$}
\allowdisplaybreaks
The actions of the generators $E_i$ in type $E_6$ corresponding to $$\bii=(4\;34\;034\;23043212340321\;5432103243054321)$$ are given by the following. Let
{
\tiny
\begin{eqnarray*}
P_6^0&=&-p_{0}^{1}+p_{3}^{1}-p_{3}^{2}+p_{4}^{1}-p_{4}^{2}+p_{5}^{1}-p_{5}^{2}+p_{4}^{3}-p_{4}^{4}+p_{3}^{5}-p_{3}^{6}+p_{0}^{3}-p_{0}^{4}+p_{3}^{7}-p_{3}^{8}+p_{4}^{5}-p_{4}^{6},\\
P_6^2&=&-p_{2}^{1}+p_{1}^{1}-p_{1}^{2}+p_{2}^{2}-p_{2}^{3}+p_{3}^{3}-p_{3}^{4}+p_{0}^{2}-p_{0}^{3}+p_{3}^{5}-p_{3}^{6}+p_{4}^{4}-p_{4}^{5}+p_{3}^{7}-p_{3}^{8}+p_{0}^{4}-p_{0}^{5}+p_{3}^{9}-p_{3}^{10}+p_{4}^{7}-p_{4}^{8},\\
P_6^3&=&-p_{3}^{1}+p_{2}^{1}-p_{2}^{2}+p_{3}^{2}-p_{3}^{3}+p_{4}^{2}-p_{4}^{3}+p_{3}^{4}-p_{3}^{5}+p_{2}^{4}-p_{2}^{5}+p_{3}^{6}-p_{3}^{7}+p_{2}^{6}-p_{2}^{7}+p_{3}^{8}-p_{3}^{9}+p_{4}^{6}-p_{4}^{7},\\
P_6^4&=&-p_{4}^{1}+p_{3}^{1}-p_{3}^{2}+p_{0}^{1}-p_{0}^{2}+p_{3}^{3}-p_{3}^{4}+p_{2}^{3}-p_{2}^{4}+p_{1}^{3}-p_{1}^{4}+p_{2}^{5}-p_{2}^{6},\\
P_6^5&=&-p_{5}^{1}+p_{4}^{1}-p_{4}^{2}+p_{3}^{2}-p_{3}^{3}+p_{2}^{2}-p_{2}^{3}+p_{1}^{2}-p_{1}^{3}.
\end{eqnarray*}
}
Then we have
{\tiny
\begin{eqnarray*}
E_0&=&[u_{0}^{1}-u_{3}^{1}]e(-p_{0}^{1})+[u_{3}^{2}-u_{4}^{1}]e(P_6^0(...p_{3}^{2}))+[u_{4}^{2}-u_{5}^{1}]e(P_6^0(...p_{4}^{2}))+[u_{5}^{2}-u_{4}^{3}]e(P_6^0(...p_{5}^{2}))\\
&&+[u_{4}^{4}-u_{3}^{5}]e(P_6^0(...p_{4}^{4}))+[u_{3}^{6}-u_{0}^{3}]e(P_6^0(...p_{3}^{6}))+[u_{0}^{4}-u_{3}^{7}]e(P_6^0(...p_{0}^{4}))+[u_{3}^{8}-u_{4}^{5}]e(P_6^0(...p_{3}^{8}))+[u_{4}^{6}]e(P_6^0),\\
E_1&=&[u_{1}^{1}]e(-p_{1}^{1}),\\
E_2&=&[u_{2}^{1}-u_{1}^{1}]e(-p_{2}^{1})+[u_{1}^{2}-u_{2}^{2}]e(P_6^2(...p_{1}^{2}))+[u_{2}^{3}-u_{3}^{3}]e(P_6^2(...p_{2}^{3}))+[u_{3}^{4}-u_{0}^{2}]e(P_6^2(...p_{3}^{4}))\\
&&+[u_{0}^{3}-u_{3}^{5}]e(P_6^2(...p_{0}^{3}))+[u_{3}^{6}-u_{4}^{4}]e(P_6^2(...p_{3}^{6}))+[u_{4}^{5}-u_{3}^{7}]e(P_6^2(...p_{4}^{5}))+[u_{3}^{8}-u_{0}^{4}]e(P_6^2(...p_{3}^{8}))\\
&&+[u_{0}^{5}-u_{3}^{9}]e(P_6^2(...p_{0}^{5}))+[u_{3}^{10}-u_{4}^{7}]e(P_6^2(...p_{3}^{10}))+[u_{4}^{8}]e(P_6^2),\\
E_3&=&[u_{3}^{1}-u_{2}^{1}]e(-p_{3}^{1})+[u_{2}^{2}-u_{3}^{2}]e(P_6^3(...p_{2}^{2}))+[u_{3}^{3}-u_{4}^{2}]e(P_6^3(...p_{3}^{3}))+[u_{4}^{3}-u_{3}^{4}]e(P_6^3(...p_{4}^{3}))\\
&&+[u_{3}^{5}-u_{2}^{4}]e(P_6^3(...p_{3}^{5}))+[u_{2}^{5}-u_{3}^{6}]e(P_6^3(...p_{2}^{5}))+[u_{3}^{7}-u_{2}^{6}]e(P_6^3(...p_{3}^{7}))+[u_{2}^{7}-u_{3}^{8}]e(P_6^3(...p_{2}^{7}))\\
&&+[u_{3}^{9}-u_{4}^{6}]e(P_6^3(...p_{3}^{9}))+[u_{4}^{7}]e(P_6^3),\\
E_4&=&[u_{4}^{1}-u_{3}^{1}]e(-p_{4}^{1})+[u_{3}^{2}-u_{0}^{1}]e(P_6^4(...p_{3}^{2}))+[u_{0}^{2}-u_{3}^{3}]e(P_6^4(...p_{0}^{2}))+[u_{3}^{4}-u_{2}^{3}]e(P_6^4(...p_{3}^{4}))\\
&&+[u_{2}^{4}-u_{1}^{3}]e(P_6^4(...p_{2}^{4}))+[u_{1}^{4}-u_{2}^{5}]e(P_6^4(...p_{1}^{4}))+[u_{2}^{6}]e(P_6^4),\\
E_5&=&[u_{5}^{1}-u_{4}^{1}]e(-p_{5}^{1})+[u_{4}^{2}-u_{3}^{2}]e(P_6^5(...p_{4}^{2}))+[u_{3}^{3}-u_{2}^{2}]e(P_6^5(...p_{3}^{3}))+[u_{2}^{3}-u_{1}^{2}]e(P_6^5(...p_{2}^{3}))+[u_{1}^{3}]e(P_6^5).
\end{eqnarray*}
}

%====================================================================================

\subsection{Positive representations in type $E_7$}
The actions of the generators $E_i$ in type $E_7$ corresponding to $$\bii=(4\;34\;034\;230432\;12340321\;5432103243054321\;65432034561234503423012345)$$ are given by the following. Let
{\tiny
\begin{eqnarray*}
P_7^0&=&-p_{0}^{1}+p_{3}^{1}-p_{3}^{2}+p_{2}^{1}-p_{2}^{2}+p_{1}^{1}-p_{1}^{2}+p_{2}^{3}-p_{2}^{4}+p_{3}^{5}-p_{3}^{6}+p_{0}^{3}-p_{0}^{4}+p_{3}^{7}-p_{3}^{8}+p_{4}^{6}-p_{4}^{7}+p_{5}^{5}-p_{5}^{6}+p_{4}^{8}-p_{4}^{9}\\
&&+p_{3}^{11}-p_{3}^{12}+p_{0}^{6}-p_{0}^{7}+p_{3}^{13}-p_{3}^{14}+p_{4}^{10}-p_{4}^{11},\\
P_7^1&=&-p_{1}^{1}+p_{2}^{1}-p_{2}^{2}+p_{3}^{2}-p_{3}^{3}+p_{4}^{2}-p_{4}^{3}+p_{5}^{2}-p_{5}^{3}+p_{6}^{2}-p_{6}^{3}+p_{5}^{4}-p_{5}^{5}+p_{4}^{6}-p_{4}^{7}+p_{3}^{8}-p_{3}^{9}+p_{2}^{6}-p_{2}^{7}+p_{1}^{4}-p_{1}^{5},\\
P_7^2&=&-p_{2}^{1}+p_{3}^{1}-p_{3}^{2}+p_{0}^{1}-p_{0}^{2}+p_{3}^{3}-p_{3}^{4}+p_{4}^{3}-p_{4}^{4}+p_{5}^{3}-p_{5}^{4}+p_{4}^{5}-p_{4}^{6}+p_{3}^{7}-p_{3}^{8}+p_{0}^{4}-p_{0}^{5}+p_{3}^{9}-p_{3}^{10}+p_{2}^{7}\\
&&-p_{2}^{8}+p_{1}^{5}-p_{1}^{6}+p_{2}^{9}-p_{2}^{10},\\
P_7^3&=&-p_{3}^{1}+p_{4}^{1}-p_{4}^{2}+p_{3}^{2}-p_{3}^{3}+p_{2}^{2}-p_{2}^{3}+p_{3}^{4}-p_{3}^{5}+p_{4}^{4}-p_{4}^{5}+p_{3}^{6}-p_{3}^{7}+p_{2}^{5}-p_{2}^{6}+p_{3}^{8}-p_{3}^{9}+p_{4}^{7}-p_{4}^{8}+p_{3}^{10}-p_{3}^{11}\\
&&+p_{2}^{8}-p_{2}^{9}+p_{3}^{12}-p_{3}^{13}+p_{2}^{10}-p_{2}^{11}+p_{3}^{14}-p_{3}^{15}+p_{4}^{11}-p_{4}^{12},\\
P_7^4&=&-p_{4}^{1}+p_{5}^{1}-p_{5}^{2}+p_{4}^{2}-p_{4}^{3}+p_{3}^{3}-p_{3}^{4}+p_{0}^{2}-p_{0}^{3}+p_{3}^{5}-p_{3}^{6}+p_{2}^{4}-p_{2}^{5}+p_{1}^{3}-p_{1}^{4}+p_{2}^{6}-p_{2}^{7}+p_{3}^{9}-p_{3}^{10}+p_{0}^{5}\\
&&-p_{0}^{6}+p_{3}^{11}-p_{3}^{12}+p_{4}^{9}-p_{4}^{10}+p_{3}^{13}-p_{3}^{14}+p_{0}^{7}-p_{0}^{8}+p_{3}^{15}-p_{3}^{16}+p_{4}^{12}-p_{4}^{13},\\
P_7^5&=&-p_{5}^{1}+p_{6}^{1}-p_{6}^{2}+p_{5}^{2}-p_{5}^{3}+p_{4}^{3}-p_{4}^{4}+p_{3}^{4}-p_{3}^{5}+p_{2}^{3}-p_{2}^{4}+p_{1}^{2}-p_{1}^{3}.\\
\end{eqnarray*}
}
Then we have
{\tiny
\begin{eqnarray*}
E_0&=&[u_{0}^{1}-u_{3}^{1}]e(-p_{0}^{1})+[u_{3}^{2}-u_{2}^{1}]e(P_7^0(...p_{3}^{2}))+[u_{2}^{2}-u_{1}^{1}]e(P_7^0(...p_{2}^{2}))+[u_{1}^{2}-u_{2}^{3}]e(P_7^0(...p_{1}^{2}))\\
&&+[u_{2}^{4}-u_{3}^{5}]e(P_7^0(...p_{2}^{4}))+[u_{3}^{6}-u_{0}^{3}]e(P_7^0(...p_{3}^{6}))+[u_{0}^{4}-u_{3}^{7}]e(P_7^0(...p_{0}^{4}))+[u_{3}^{8}-u_{4}^{6}]e(P_7^0(...p_{3}^{8}))\\
&&+[u_{4}^{7}-u_{5}^{5}]e(P_7^0(...p_{4}^{7}))+[u_{5}^{6}-u_{4}^{8}]e(P_7^0(...p_{5}^{6}))+[u_{4}^{9}-u_{3}^{11}]e(P_7^0(...p_{4}^{9}))+[u_{3}^{12}-u_{0}^{6}]e(P_7^0(...p_{3}^{12}))\\
&&+[u_{0}^{7}-u_{3}^{13}]e(P_7^0(...p_{0}^{7}))+[u_{3}^{14}-u_{4}^{10}]e(P_7^0(...p_{3}^{14}))+[u_{4}^{11}]e(P_7^0),\\
E_1&=&[u_{1}^{1}-u_{2}^{1}]e(-p_{1}^{1})+[u_{2}^{2}-u_{3}^{2}]e(P_7^1(...p_{2}^{2}))+[u_{3}^{3}-u_{4}^{2}]e(P_7^1(...p_{3}^{3}))+[u_{4}^{3}-u_{5}^{2}]e(P_7^1(...p_{4}^{3}))\\
&&+[u_{5}^{3}-u_{6}^{2}]e(P_7^1(...p_{5}^{3}))+[u_{6}^{3}-u_{5}^{4}]e(P_7^1(...p_{6}^{3}))+[u_{5}^{5}-u_{4}^{6}]e(P_7^1(...p_{5}^{5}))+[u_{4}^{7}-u_{3}^{8}]e(P_7^1(...p_{4}^{7}))\\
&&+[u_{3}^{9}-u_{2}^{6}]e(P_7^1(...p_{3}^{9}))+[u_{2}^{7}-u_{1}^{4}]e(P_7^1(...p_{2}^{7}))+[u_{1}^{5}]e(P_7^1),\\
E_2&=&[u_{2}^{1}-u_{3}^{1}]e(-p_{2}^{1})+[u_{3}^{2}-u_{0}^{1}]e(P_7^2(...p_{3}^{2}))+[u_{0}^{2}-u_{3}^{3}]e(P_7^2(...p_{0}^{2}))+[u_{3}^{4}-u_{4}^{3}]e(P_7^2(...p_{3}^{4}))\\
&&+[u_{4}^{4}-u_{5}^{3}]e(P_7^2(...p_{4}^{4}))+[u_{5}^{4}-u_{4}^{5}]e(P_7^2(...p_{5}^{4}))+[u_{4}^{6}-u_{3}^{7}]e(P_7^2(...p_{4}^{6}))+[u_{3}^{8}-u_{0}^{4}]e(P_7^2(...p_{3}^{8}))\\
&&+[u_{0}^{5}-u_{3}^{9}]e(P_7^2(...p_{0}^{5}))+[u_{3}^{10}-u_{2}^{7}]e(P_7^2(...p_{3}^{10}))+[u_{2}^{8}-u_{1}^{5}]e(P_7^2(...p_{2}^{8}))+[u_{1}^{6}-u_{2}^{9}]e(P_7^2(...p_{1}^{6}))+[u_{2}^{10}]e(P_7^2),\\
E_3&=&[u_{3}^{1}-u_{4}^{1}]e(-p_{3}^{1})+[u_{4}^{2}-u_{3}^{2}]e(P_7^3(...p_{4}^{2}))+[u_{3}^{3}-u_{2}^{2}]e(P_7^3(...p_{3}^{3}))+[u_{2}^{3}-u_{3}^{4}]e(P_7^3(...p_{2}^{3}))+[u_{3}^{5}-u_{4}^{4}]e(P_7^3(...p_{3}^{5}))\\
&&+[u_{4}^{5}-u_{3}^{6}]e(P_7^3(...p_{4}^{5}))+[u_{3}^{7}-u_{2}^{5}]e(P_7^3(...p_{3}^{7}))+[u_{2}^{6}-u_{3}^{8}]e(P_7^3(...p_{2}^{6}))+[u_{3}^{9}-u_{4}^{7}]e(P_7^3(...p_{3}^{9}))\\
&&+[u_{4}^{8}-u_{3}^{10}]e(P_7^3(...p_{4}^{8}))+[u_{3}^{11}-u_{2}^{8}]e(P_7^3(...p_{3}^{11}))+[u_{2}^{9}-u_{3}^{12}]e(P_7^3(...p_{2}^{9}))+[u_{3}^{13}-u_{2}^{10}]e(P_7^3(...p_{3}^{13}))\\
&&+[u_{2}^{11}-u_{3}^{14}]e(P_7^3(...p_{2}^{11}))+[u_{3}^{15}-u_{4}^{11}]e(P_7^3(...p_{3}^{15}))+[u_{4}^{12}]e(P_7^3),\\
E_4&=&[u_{4}^{1}-u_{5}^{1}]e(-p_{4}^{1})+[u_{5}^{2}-u_{4}^{2}]e(P_7^4(...p_{5}^{2}))+[u_{4}^{3}-u_{3}^{3}]e(P_7^4(...p_{4}^{3}))+[u_{3}^{4}-u_{0}^{2}]e(P_7^4(...p_{3}^{4}))\\
&&+[u_{0}^{3}-u_{3}^{5}]e(P_7^4(...p_{0}^{3}))+[u_{3}^{6}-u_{2}^{4}]e(P_7^4(...p_{3}^{6}))+[u_{2}^{5}-u_{1}^{3}]e(P_7^4(...p_{2}^{5}))+[u_{1}^{4}-u_{2}^{6}]e(P_7^4(...p_{1}^{4}))\\
&&+[u_{2}^{7}-u_{3}^{9}]e(P_7^4(...p_{2}^{7}))+[u_{3}^{10}-u_{0}^{5}]e(P_7^4(...p_{3}^{10}))+[u_{0}^{6}-u_{3}^{11}]e(P_7^4(...p_{0}^{6}))+[u_{3}^{12}-u_{4}^{9}]e(P_7^4(...p_{3}^{12}))\\
&&+[u_{4}^{10}-u_{3}^{13}]e(P_7^4(...p_{4}^{10}))+[u_{3}^{14}-u_{0}^{7}]e(P_7^4(...p_{3}^{14}))+[u_{0}^{8}-u_{3}^{15}]e(P_7^4(...p_{0}^{8}))+[u_{3}^{16}-u_{4}^{12}]e(P_7^4(...p_{3}^{16}))\\
&&+[u_{4}^{13}]e(P_7^4),\\
E_5&=&[u_{5}^{1}-u_{6}^{1}]e(-p_{5}^{1})+[u_{6}^{2}-u_{5}^{2}]e(P_7^5(...p_{6}^{2}))+[u_{5}^{3}-u_{4}^{3}]e(P_7^5(...p_{5}^{3}))+[u_{4}^{4}-u_{3}^{4}]e(P_7^5(...p_{4}^{4}))+[u_{3}^{5}-u_{2}^{3}]e(P_7^5(...p_{3}^{5}))\\
&&+[u_{2}^{4}-u_{1}^{2}]e(P_7^5(...p_{2}^{4}))+[u_{1}^{3}]e(P_7^5),\\
E_6&=&[u_{6}^{1}]e(-p_{6}^{1}).\\
\end{eqnarray*}
}
%====================================================================================

\subsection{Positive representations in type $E_8$}
The actions of the generators $E_i$ in type $E_8$ corresponding to
\Eqn{
\bii&=(4\;34\;034\;230432\;12340321\;5432103243054321\;654320345612345034230123456\\
&765432103243546503423012345676543203456123450342301234567)
}
are given by the following. Let
{\tiny
\begin{eqnarray*}
P_8^0&=&-p_{0}^{1}+p_{3}^{1}-p_{3}^{2}+p_{2}^{1}-p_{2}^{2}+p_{1}^{1}-p_{1}^{2}+p_{2}^{3}-p_{2}^{4}+p_{3}^{5}-p_{3}^{6}+p_{0}^{3}-p_{0}^{4}+p_{3}^{7}-p_{3}^{8}+p_{2}^{5}-p_{2}^{6}+p_{1}^{3}-p_{1}^{4}+p_{2}^{7}-p_{2}^{8}\\
&&+p_{3}^{11}-p_{3}^{12}+p_{0}^{6}-p_{0}^{7}+p_{3}^{13}-p_{3}^{14}+p_{2}^{9}-p_{2}^{10}+p_{1}^{5}-p_{1}^{6}+p_{2}^{11}-p_{2}^{12}+p_{3}^{17}-p_{3}^{18}+p_{0}^{9}-p_{0}^{10}+p_{3}^{19}-p_{3}^{20}\\
&&+p_{4}^{16}-p_{4}^{17}+p_{5}^{13}-p_{5}^{14}+p_{4}^{18}-p_{4}^{19}+p_{3}^{23}-p_{3}^{24}+p_{0}^{12}-p_{0}^{13}+p_{3}^{25}-p_{3}^{26}+p_{4}^{20}-p_{4}^{21},\\
P_8^1&=&-p_{1}^{1}+p_{2}^{1}-p_{2}^{2}+p_{3}^{2}-p_{3}^{3}+p_{4}^{2}-p_{4}^{3}+p_{5}^{2}-p_{5}^{3}+p_{6}^{2}-p_{6}^{3}+p_{5}^{4}-p_{5}^{5}+p_{6}^{4}-p_{6}^{5}+p_{5}^{6}-p_{5}^{7}+p_{4}^{8}-p_{4}^{9}+p_{3}^{10}-p_{3}^{11}\\
&&+p_{2}^{7}-p_{2}^{8}+p_{1}^{4}-p_{1}^{5}+p_{2}^{9}-p_{2}^{10}+p_{3}^{14}-p_{3}^{15}+p_{4}^{12}-p_{4}^{13}+p_{5}^{10}-p_{5}^{11}+p_{6}^{8}-p_{6}^{9}+p_{5}^{12}-p_{5}^{13}+p_{4}^{16}-p_{4}^{17}+p_{3}^{20}\\
&&-p_{3}^{21}+p_{2}^{14}-p_{2}^{15}+p_{1}^{8}-p_{1}^{9},\\
P_8^2&=&-p_{2}^{1}+p_{3}^{1}-p_{3}^{2}+p_{0}^{1}-p_{0}^{2}+p_{3}^{3}-p_{3}^{4}+p_{4}^{3}-p_{4}^{4}+p_{5}^{3}-p_{5}^{4}+p_{4}^{5}-p_{4}^{6}+p_{5}^{5}-p_{5}^{6}+p_{4}^{7}-p_{4}^{8}+p_{3}^{9}-p_{3}^{10}+p_{0}^{5}-p_{0}^{6}\\
&&+p_{3}^{11}-p_{3}^{12}+p_{2}^{8}-p_{2}^{9}+p_{3}^{13}-p_{3}^{14}+p_{0}^{7}-p_{0}^{8}+p_{3}^{15}-p_{3}^{16}+p_{4}^{13}-p_{4}^{14}+p_{5}^{11}-p_{5}^{12}+p_{4}^{15}-p_{4}^{16}+p_{3}^{19}-p_{3}^{20}\\
&&+p_{0}^{10}-p_{0}^{11}+p_{3}^{21}-p_{3}^{22}+p_{2}^{15}-p_{2}^{16}+p_{1}^{9}-p_{1}^{10}+p_{2}^{17}-p_{2}^{18},\\
P_8^3&=&-p_{3}^{1}+p_{4}^{1}-p_{4}^{2}+p_{3}^{2}-p_{3}^{3}+p_{2}^{2}-p_{2}^{3}+p_{3}^{4}-p_{3}^{5}+p_{4}^{4}-p_{4}^{5}+p_{3}^{6}-p_{3}^{7}+p_{4}^{6}-p_{4}^{7}+p_{3}^{8}-p_{3}^{9}+p_{2}^{6}-p_{2}^{7}+p_{3}^{10}-p_{3}^{11}\\
&&+p_{4}^{9}-p_{4}^{10}+p_{3}^{12}-p_{3}^{13}+p_{4}^{11}-p_{4}^{12}+p_{3}^{14}-p_{3}^{15}+p_{2}^{10}-p_{2}^{11}+p_{3}^{16}-p_{3}^{17}+p_{4}^{14}-p_{4}^{15}+p_{3}^{18}-p_{3}^{19}+p_{2}^{13}-p_{2}^{14}\\
&&+p_{3}^{20}-p_{3}^{21}+p_{4}^{17}-p_{4}^{18}+p_{3}^{22}-p_{3}^{23}+p_{2}^{16}-p_{2}^{17}+p_{3}^{24}-p_{3}^{25}+p_{2}^{18}-p_{2}^{19}+p_{3}^{26}-p_{3}^{27}+p_{4}^{21}-p_{4}^{22},\\
P_8^4&=&-p_{4}^{1}+p_{5}^{1}-p_{5}^{2}+p_{4}^{2}-p_{4}^{3}+p_{3}^{3}-p_{3}^{4}+p_{0}^{2}-p_{0}^{3}+p_{3}^{5}-p_{3}^{6}+p_{2}^{4}-p_{2}^{5}+p_{3}^{7}-p_{3}^{8}+p_{0}^{4}-p_{0}^{5}+p_{3}^{9}-p_{3}^{10}+p_{4}^{8}-p_{4}^{9}\\
&&+p_{5}^{7}-p_{5}^{8}+p_{4}^{10}-p_{4}^{11}+p_{5}^{9}-p_{5}^{10}+p_{4}^{12}-p_{4}^{13}+p_{3}^{15}-p_{3}^{16}+p_{0}^{8}-p_{0}^{9}+p_{3}^{17}-p_{3}^{18}+p_{2}^{12}-p_{2}^{13}+p_{1}^{7}-p_{1}^{8}+p_{2}^{14}\\
&&-p_{2}^{15}+p_{3}^{21}-p_{3}^{22}+p_{0}^{11}-p_{0}^{12}+p_{3}^{23}-p_{3}^{24}+p_{4}^{19}-p_{4}^{20}+p_{3}^{25}-p_{3}^{26}+p_{0}^{13}-p_{0}^{14}+p_{3}^{27}-p_{3}^{28}+p_{4}^{22}-p_{4}^{23},\\
P_8^5&=&-p_{5}^{1}+p_{6}^{1}-p_{6}^{2}+p_{5}^{2}-p_{5}^{3}+p_{4}^{3}-p_{4}^{4}+p_{3}^{4}-p_{3}^{5}+p_{2}^{3}-p_{2}^{4}+p_{1}^{2}-p_{1}^{3}+p_{2}^{5}-p_{2}^{6}+p_{3}^{8}-p_{3}^{9}+p_{4}^{7}-p_{4}^{8}+p_{5}^{6}-p_{5}^{7}\\
&&+p_{6}^{5}-p_{6}^{6}+p_{5}^{8}-p_{5}^{9}+p_{6}^{7}-p_{6}^{8}+p_{5}^{10}-p_{5}^{11}+p_{4}^{13}-p_{4}^{14}+p_{3}^{16}-p_{3}^{17}+p_{2}^{11}-p_{2}^{12}+p_{1}^{6}-p_{1}^{7},\\
P_8^6&=&-p_6^1+p_7^1-p_7^2+p_6^2-p_6^4+p_5^3-p_5^5+p_4^4-p_4^6+p_3^5-p_3^7+p_2^4-p_2^5+p_0^3-p_0^4+p_3^6-p_3^8+p_4^5-p_4^7+p_5^4-p_5^6\\
&&+p_6^3-p_6^5+p_6^6-p_6^7.\\
\end{eqnarray*}
}
Then we have
{\tiny
\begin{eqnarray*}
E_0&=&[u_{0}^{1}-u_{3}^{1}]e(-p_{0}^{1})+[u_{3}^{2}-u_{2}^{1}]e(P_8^0(...p_{3}^{2}))+[u_{2}^{2}-u_{1}^{1}]e(P_8^0(...p_{2}^{2}))+[u_{1}^{2}-u_{2}^{3}]e(P_8^0(...p_{1}^{2}))\\
&&+[u_{2}^{4}-u_{3}^{5}]e(P_8^0(...p_{2}^{4}))+[u_{3}^{6}-u_{0}^{3}]e(P_8^0(...p_{3}^{6}))+[u_{0}^{4}-u_{3}^{7}]e(P_8^0(...p_{0}^{4}))+[u_{3}^{8}-u_{2}^{5}]e(P_8^0(...p_{3}^{8}))\\
&&+[u_{2}^{6}-u_{1}^{3}]e(P_8^0(...p_{2}^{6}))+[u_{1}^{4}-u_{2}^{7}]e(P_8^0(...p_{1}^{4}))+[u_{2}^{8}-u_{3}^{11}]e(P_8^0(...p_{2}^{8}))+[u_{3}^{12}-u_{0}^{6}]e(P_8^0(...p_{3}^{12}))\\
&&+[u_{0}^{7}-u_{3}^{13}]e(P_8^0(...p_{0}^{7}))+[u_{3}^{14}-u_{2}^{9}]e(P_8^0(...p_{3}^{14}))+[u_{2}^{10}-u_{1}^{5}]e(P_8^0(...p_{2}^{10}))+[u_{1}^{6}-u_{2}^{11}]e(P_8^0(...p_{1}^{6}))\\
&&+[u_{2}^{12}-u_{3}^{17}]e(P_8^0(...p_{2}^{12}))+[u_{3}^{18}-u_{0}^{9}]e(P_8^0(...p_{3}^{18}))+[u_{0}^{10}-u_{3}^{19}]e(P_8^0(...p_{0}^{10}))+[u_{3}^{20}-u_{4}^{16}]e(P_8^0(...p_{3}^{20}))\\
&&+[u_{4}^{17}-u_{5}^{13}]e(P_8^0(...p_{4}^{17}))+[u_{5}^{14}-u_{4}^{18}]e(P_8^0(...p_{5}^{14}))+[u_{4}^{19}-u_{3}^{23}]e(P_8^0(...p_{4}^{19}))+[u_{3}^{24}-u_{0}^{12}]e(P_8^0(...p_{3}^{24}))\\
&&+[u_{0}^{13}-u_{3}^{25}]e(P_8^0(...p_{0}^{13}))+[u_{3}^{26}-u_{4}^{20}]e(P_8^0(...p_{3}^{26}))+[u_{4}^{21}]e(P_8^0),\\
E_1&=&[u_{1}^{1}-u_{2}^{1}]e(-p_{1}^{1})+[u_{2}^{2}-u_{3}^{2}]e(P_8^1(...p_{2}^{2}))+[u_{3}^{3}-u_{4}^{2}]e(P_8^1(...p_{3}^{3}))+[u_{4}^{3}-u_{5}^{2}]e(P_8^1(...p_{4}^{3}))\\
&&+[u_{5}^{3}-u_{6}^{2}]e(P_8^1(...p_{5}^{3}))+[u_{6}^{3}-u_{5}^{4}]e(P_8^1(...p_{6}^{3}))+[u_{5}^{5}-u_{6}^{4}]e(P_8^1(...p_{5}^{5}))+[u_{6}^{5}-u_{5}^{6}]e(P_8^1(...p_{6}^{5}))\\
&&+[u_{5}^{7}-u_{4}^{8}]e(P_8^1(...p_{5}^{7}))+[u_{4}^{9}-u_{3}^{10}]e(P_8^1(...p_{4}^{9}))+[u_{3}^{11}-u_{2}^{7}]e(P_8^1(...p_{3}^{11}))+[u_{2}^{8}-u_{1}^{4}]e(P_8^1(...p_{2}^{8}))\\
&&+[u_{1}^{5}-u_{2}^{9}]e(P_8^1(...p_{1}^{5}))+[u_{2}^{10}-u_{3}^{14}]e(P_8^1(...p_{2}^{10}))+[u_{3}^{15}-u_{4}^{12}]e(P_8^1(...p_{3}^{15}))+[u_{4}^{13}-u_{5}^{10}]e(P_8^1(...p_{4}^{13}))\\
&&+[u_{5}^{11}-u_{6}^{8}]e(P_8^1(...p_{5}^{11}))+[u_{6}^{9}-u_{5}^{12}]e(P_8^1(...p_{6}^{9}))+[u_{5}^{13}-u_{4}^{16}]e(P_8^1(...p_{5}^{13}))+[u_{4}^{17}-u_{3}^{20}]e(P_8^1(...p_{4}^{17}))\\
&&+[u_{3}^{21}-u_{2}^{14}]e(P_8^1(...p_{3}^{21}))+[u_{2}^{15}-u_{1}^{8}]e(P_8^1(...p_{2}^{15}))+[u_{1}^{9}]e(P_8^1),\\
E_2&=&
[u_{2}^{1}-u_{3}^{1}]e(-p_{2}^{1})+[u_{3}^{2}-u_{0}^{1}]e(P_8^2(...p_{3}^{2}))+[u_{0}^{2}-u_{3}^{3}]e(P_8^2(...p_{0}^{2}))+[u_{3}^{4}-u_{4}^{3}]e(P_8^2(...p_{3}^{4}))\\
&&+[u_{4}^{4}-u_{5}^{3}]e(P_8^2(...p_{4}^{4}))+[u_{5}^{4}-u_{4}^{5}]e(P_8^2(...p_{5}^{4}))+[u_{4}^{6}-u_{5}^{5}]e(P_8^2(...p_{4}^{6}))+[u_{5}^{6}-u_{4}^{7}]e(P_8^2(...p_{5}^{6}))\\
&&+[u_{4}^{8}-u_{3}^{9}]e(P_8^2(...p_{4}^{8}))+[u_{3}^{10}-u_{0}^{5}]e(P_8^2(...p_{3}^{10}))+[u_{0}^{6}-u_{3}^{11}]e(P_8^2(...p_{0}^{6}))+[u_{3}^{12}-u_{2}^{8}]e(P_8^2(...p_{3}^{12}))\\
&&+[u_{2}^{9}-u_{3}^{13}]e(P_8^2(...p_{2}^{9}))+[u_{3}^{14}-u_{0}^{7}]e(P_8^2(...p_{3}^{14}))+[u_{0}^{8}-u_{3}^{15}]e(P_8^2(...p_{0}^{8}))+[u_{3}^{16}-u_{4}^{13}]e(P_8^2(...p_{3}^{16}))\\
&&+[u_{4}^{14}-u_{5}^{11}]e(P_8^2(...p_{4}^{14}))+[u_{5}^{12}-u_{4}^{15}]e(P_8^2(...p_{5}^{12}))+[u_{4}^{16}-u_{3}^{19}]e(P_8^2(...p_{4}^{16}))+[u_{3}^{20}-u_{0}^{10}]e(P_8^2(...p_{3}^{20}))\\
&&+[u_{0}^{11}-u_{3}^{21}]e(P_8^2(...p_{0}^{11}))+[u_{3}^{22}-u_{2}^{15}]e(P_8^2(...p_{3}^{22}))+[u_{2}^{16}-u_{1}^{9}]e(P_8^2(...p_{2}^{16}))+[u_{1}^{10}-u_{2}^{17}]e(P_8^2(...p_{1}^{10}))\\
&&+[u_{2}^{18}]e(P_8^2),\\
E_3&=&[u_{3}^{1}-u_{4}^{1}]e(-p_{3}^{1})+[u_{4}^{2}-u_{3}^{2}]e(P_8^3(...p_{4}^{2}))+[u_{3}^{3}-u_{2}^{2}]e(P_8^3(...p_{3}^{3}))+[u_{2}^{3}-u_{3}^{4}]e(P_8^3(...p_{2}^{3}))\\
&&+[u_{3}^{5}-u_{4}^{4}]e(P_8^3(...p_{3}^{5}))+[u_{4}^{5}-u_{3}^{6}]e(P_8^3(...p_{4}^{5}))+[u_{3}^{7}-u_{4}^{6}]e(P_8^3(...p_{3}^{7}))+[u_{4}^{7}-u_{3}^{8}]e(P_8^3(...p_{4}^{7}))\\
&&+[u_{3}^{9}-u_{2}^{6}]e(P_8^3(...p_{3}^{9}))+[u_{2}^{7}-u_{3}^{10}]e(P_8^3(...p_{2}^{7}))+[u_{3}^{11}-u_{4}^{9}]e(P_8^3(...p_{3}^{11}))+[u_{4}^{10}-u_{3}^{12}]e(P_8^3(...p_{4}^{10}))\\
&&+[u_{3}^{13}-u_{4}^{11}]e(P_8^3(...p_{3}^{13}))+[u_{4}^{12}-u_{3}^{14}]e(P_8^3(...p_{4}^{12}))+[u_{3}^{15}-u_{2}^{10}]e(P_8^3(...p_{3}^{15}))+[u_{2}^{11}-u_{3}^{16}]e(P_8^3(...p_{2}^{11}))\\
&&+[u_{3}^{17}-u_{4}^{14}]e(P_8^3(...p_{3}^{17}))+[u_{4}^{15}-u_{3}^{18}]e(P_8^3(...p_{4}^{15}))+[u_{3}^{19}-u_{2}^{13}]e(P_8^3(...p_{3}^{19}))+[u_{2}^{14}-u_{3}^{20}]e(P_8^3(...p_{2}^{14}))\\
&&+[u_{3}^{21}-u_{4}^{17}]e(P_8^3(...p_{3}^{21}))+[u_{4}^{18}-u_{3}^{22}]e(P_8^3(...p_{4}^{18}))+[u_{3}^{23}-u_{2}^{16}]e(P_8^3(...p_{3}^{23}))+[u_{2}^{17}-u_{3}^{24}]e(P_8^3(...p_{2}^{17}))\\
&&+[u_{3}^{25}-u_{2}^{18}]e(P_8^3(...p_{3}^{25}))+[u_{2}^{19}-u_{3}^{26}]e(P_8^3(...p_{2}^{19}))+[u_{3}^{27}-u_{4}^{21}]e(P_8^3(...p_{3}^{27}))+[u_{4}^{22}]e(P_8^3),\\
E_4&=&[u_{4}^{1}-u_{5}^{1}]e(-p_{4}^{1})+[u_{5}^{2}-u_{4}^{2}]e(P_8^4(...p_{5}^{2}))+[u_{4}^{3}-u_{3}^{3}]e(P_8^4(...p_{4}^{3}))+[u_{3}^{4}-u_{0}^{2}]e(P_8^4(...p_{3}^{4}))\\
&&+[u_{0}^{3}-u_{3}^{5}]e(P_8^4(...p_{0}^{3}))+[u_{3}^{6}-u_{2}^{4}]e(P_8^4(...p_{3}^{6}))+[u_{2}^{5}-u_{3}^{7}]e(P_8^4(...p_{2}^{5}))+[u_{3}^{8}-u_{0}^{4}]e(P_8^4(...p_{3}^{8}))\\
&&+[u_{0}^{5}-u_{3}^{9}]e(P_8^4(...p_{0}^{5}))+[u_{3}^{10}-u_{4}^{8}]e(P_8^4(...p_{3}^{10}))+[u_{4}^{9}-u_{5}^{7}]e(P_8^4(...p_{4}^{9}))+[u_{5}^{8}-u_{4}^{10}]e(P_8^4(...p_{5}^{8}))\\
&&+[u_{4}^{11}-u_{5}^{9}]e(P_8^4(...p_{4}^{11}))+[u_{5}^{10}-u_{4}^{12}]e(P_8^4(...p_{5}^{10}))+[u_{4}^{13}-u_{3}^{15}]e(P_8^4(...p_{4}^{13}))+[u_{3}^{16}-u_{0}^{8}]e(P_8^4(...p_{3}^{16}))\\
&&+[u_{0}^{9}-u_{3}^{17}]e(P_8^4(...p_{0}^{9}))+[u_{3}^{18}-u_{2}^{12}]e(P_8^4(...p_{3}^{18}))+[u_{2}^{13}-u_{1}^{7}]e(P_8^4(...p_{2}^{13}))+[u_{1}^{8}-u_{2}^{14}]e(P_8^4(...p_{1}^{8}))\\
&&+[u_{2}^{15}-u_{3}^{21}]e(P_8^4(...p_{2}^{15}))+[u_{3}^{22}-u_{0}^{11}]e(P_8^4(...p_{3}^{22}))+[u_{0}^{12}-u_{3}^{23}]e(P_8^4(...p_{0}^{12}))+[u_{3}^{24}-u_{4}^{19}]e(P_8^4(...p_{3}^{24}))\\
&&+[u_{4}^{20}-u_{3}^{25}]e(P_8^4(...p_{4}^{20}))+[u_{3}^{26}-u_{0}^{13}]e(P_8^4(...p_{3}^{26}))+[u_{0}^{14}-u_{3}^{27}]e(P_8^4(...p_{0}^{14}))+[u_{3}^{28}-u_{4}^{22}]e(P_8^4(...p_{3}^{28}))\\
&&+[u_{4}^{23}]e(P_8^4),\\
E_5&=&[u_{5}^{1}-u_{6}^{1}]e(-p_{5}^{1})+[u_{6}^{2}-u_{5}^{2}]e(P_8^5(...p_{6}^{2}))+[u_{5}^{3}-u_{4}^{3}]e(P_8^5(...p_{5}^{3}))+[u_{4}^{4}-u_{3}^{4}]e(P_8^5(...p_{4}^{4}))\\
&&+[u_{3}^{5}-u_{2}^{3}]e(P_8^5(...p_{3}^{5}))+[u_{2}^{4}-u_{1}^{2}]e(P_8^5(...p_{2}^{4}))+[u_{1}^{3}-u_{2}^{5}]e(P_8^5(...p_{1}^{3}))+[u_{2}^{6}-u_{3}^{8}]e(P_8^5(...p_{2}^{6}))\\
&&+[u_{3}^{9}-u_{4}^{7}]e(P_8^5(...p_{3}^{9}))+[u_{4}^{8}-u_{5}^{6}]e(P_8^5(...p_{4}^{8}))+[u_{5}^{7}-u_{6}^{5}]e(P_8^5(...p_{5}^{7}))+[u_{6}^{6}-u_{5}^{8}]e(P_8^5(...p_{6}^{6}))\\
&&+[u_{5}^{9}-u_{6}^{7}]e(P_8^5(...p_{5}^{9}))+[u_{6}^{8}-u_{5}^{10}]e(P_8^5(...p_{6}^{8}))+[u_{5}^{11}-u_{4}^{13}]e(P_8^5(...p_{5}^{11}))+[u_{4}^{14}-u_{3}^{16}]e(P_8^5(...p_{4}^{14}))\\
&&+[u_{3}^{17}-u_{2}^{11}]e(P_8^5(...p_{3}^{17}))+[u_{2}^{12}-u_{1}^{6}]e(P_8^5(...p_{2}^{12}))+[u_{1}^{7}]e(P_8^5),\\
E_6&=&
[u_{6}^{1}-u_{7}^{1}]e(-p_{6}^{1})+[u_{7}^{2}-u_{6}^{3}-u_{6}^{2}]e(P_8^6(...p_7^2))+[u_{6}^{4}+u_{6}^{3}-u_{5}^{4}-u_{5}^{3}]e(P_8^6(...p_6^4))\\
&&+[u_{6}^{4}-u_{6}^{2}]e(P_8^6(...p_6^4)+p_6^3-p_6^2)+[u_{5}^{5}+u_{5}^{4}-u_{4}^{5}-u_{4}^{4}]e(P_8^6(...p_5^5))+[u_{5}^{5}-u_{5}^{3}]e(P_8^6(...p_5^5)+p_5^4-p_5^3)\\
&&+[u_{4}^{6}+u_{4}^{5}-u_{3}^{6}-u_{3}^{5}]e(P_8^6(...p_4^6))+[u_{4}^{6}-u_{4}^{4}]e(P_8^6(...p_4^6)+p_4^5-p_4^4)+[u_{3}^{7}+u_{3}^{6}-u_{2}^{4}-u_{0}^{3}]e(P_8^6(...p_3^7))\\
&&+[u_{3}^{7}-u_{3}^{5}]e(P_8^6(...p_3^7)+p_3^6-p_3^5)+[u_{2}^{5}-u_{0}^{3}]e(P_8^6(...p_2^5))+[u_{0}^{4}+u_{2}^{5}-u_{3}^{7}-u_{3}^{6}]e(P_8^6(...p_0^4))\\
&&+[u_{0}^{4}-u_{2}^{4}]e(P_8^6(...p_0^4)+p_2^5-p_2^4)+[u_3^8+u_{3}^{7}-u_{4}^{6}-u_{4}^{5}]e(P_8^6(...p_3^8))+[u_{3}^{8}-u_{3}^{6}]e(P_8^6(...p_3^8)+p_3^7-p_3^6)\\
&&+[u_{4}^{7}+u_{4}^{6}-u_{5}^{5}-u_{5}^{4}]e(P_8^6(...p_4^7))+[u_{4}^{7}-u_{4}^{5}]e(P_8^6(...p_4^7)+p_4^6-p_4^5)+[u_{5}^{6}+u_{5}^{5}-u_{6}^{4}-u_{6}^{3}]e(P_8^6(...p_5^6))\\
&&+[u_{5}^{6}-u_{5}^{4}]e(P_8^6(...p_5^6)+p_5^5-p_5^4)+[u_{6}^{5}+u_{6}^{4}-u_{7}^{2}]e(P_8^6(...p_6^5))+[u_{6}^{5}-u_{6}^{3}]e(P_8^6(...p_6^5)+p_6^4-p_6^3)\\
&&+[u_{7}^{3}-u_{6}^{6}]e(P_8^6(...p_6^5)+p_7^2-p_7^3)+[u_{6}^{7}]e(P_8^6+p_7^2-p_7^3),\\
E_7&=&[u_{7}^{1}]e(-p_{7}^{1}).\\
\end{eqnarray*}
}
Note that the action of the generator $E_6$ in this case is quite different from the other operators.
\end{appendices}

\end{document}